\documentclass[a4paper]{amsart}
\usepackage[margin=3.3cm]{geometry}
\usepackage[utf8]{inputenc}
\usepackage[english]{babel}
\usepackage{amssymb}
\usepackage{amsmath}
\usepackage{amsthm}
\usepackage{amsfonts}
\usepackage{array}
\usepackage{comment,stmaryrd}
\usepackage{dirtytalk}
\usepackage{enumitem}
\usepackage{hyperref}
\usepackage{mathrsfs}
\usepackage{tikz}
\usetikzlibrary{arrows,decorations.markings}
\usepackage{tikz-cd, mathtools}
\usepackage{mathabx}
\usepackage{xfrac}
\usepackage[aligntableaux=center]{ytableau}

\usepackage{graphicx}
\graphicspath{ {./Images/} }

\addto\captionsenglish{}

\theoremstyle{plain}
\newtheorem{theorem}{Theorem}[section]
\newtheorem{lemma}[theorem]{Lemma}

\newtheorem{proposition}[theorem]{Proposition}
\newtheorem{corollary}[theorem]{Corollary}

\theoremstyle{definition}
\newtheorem{definition}[theorem]{Definition}
\newtheorem{example}[theorem]{Example}

\theoremstyle{remark}
\newtheorem{remark}[theorem]{Remark}

\newtheorem*{theorem*}{Theorem}
\newtheorem*{conjecture*}{Conjecture}

\DeclareMathOperator{\Ker}{Ker}
\DeclareMathOperator{\im}{Im}

\DeclareMathOperator{\id}{id}

\DeclareMathOperator{\mwt}{\underline{wt}}

\DeclareMathOperator{\Hilb}{Hilb}
\DeclareMathOperator{\Quot}{Quot}
\DeclareMathOperator{\NQV}{\mathfrak{M}}
\DeclareMathOperator{\Part}{\mathcal{P}}
\newcommand*{\SHom}{\mathcal{H}\kern -.5pt om}
\newcommand*{\SExt}{\mathcal{E}\kern -.5pt xt}

\DeclareMathOperator{\SL}{SL}
\DeclareMathOperator{\Z}{\mathbb{Z}}

\DeclareMathOperator{\N}{\mathbb{N}}
\DeclareMathOperator{\Q}{\mathbb{Q}}
\DeclareMathOperator{\R}{\mathbb{R}}
\DeclareMathOperator{\C}{\mathbb{C}}

\newcommand{\CP}{\mathcal{P}}		% set of partitions
\newcommand{\wt}{\mathrm{wt}}		% weight map
\newcommand{\BF}{\mathbf{F}}                    % better Fock space

\newcommand{\vect}[1]{|#1\rangle}

\newcommand\lukas{%\color{blue} %Lukas:\ 
}

\makeatletter
\let\@wraptoccontribs\wraptoccontribs
\makeatother

\title[Euler characteristics of affine ADE Nakajima quiver varieties]{Euler characteristics of affine ADE Nakajima quiver varieties via collapsing fibers}
\author{Lukas Bertsch}
\address{
University of Vienna, 
Austria}
\email{lukas.bertsch@univie.ac.at }

\author{\'Ad\'am Gyenge}
\address{
Department of Algebra and Geometry, Institute of Mathematics, 
Budapest University of Technology and Economics, 
M\H{u}egyetem rakpart 3, H-1111, Budapest, 
Hungary}
\email{Gyenge.Adam@ttk.bme.hu}

\author{Bal\'{a}zs Szendr\H{o}i}
\address{
University of Vienna, 
Austria}
\email{balazs.szendroi@univie.ac.at}

\subjclass[2020]{Primary 14D20; Secondary 	14D23, 14J17}
\keywords{quiver variety, stability condition, nongeneric, generating series}

\begin{document}

\begin{abstract}
We prove a universal substitution formula that compares generating series of Euler characteristics of Nakajima quiver varieties associated with affine ADE diagrams at generic and at certain nongeneric stability conditions via a study of collapsing fibres in the associated variation of GIT map, unifying and generalising earlier results of the last two authors with Némethi and of Nakajima. As a special case, we compute generating series of Euler characteristics of noncommutative Quot schemes of Kleinian orbifolds. In type A and rank 1, we give a second, combinatorial proof of our substitution formula, using torus localisation and partition enumeration. This gives a combinatorial model of the fibers of the variation of GIT map, and also leads to relations between our results and the representation theory of the affine and finite Lie algebras in type A.
\end{abstract}

\maketitle

\section{Introduction}\label{sec:Intro}

Nakajima quiver varieties are hyper-K\"ahler quotients that often represent moduli functors. Some examples include the instanton moduli spaces of ALE spaces \cite{nakajima1994instantons}, noncommutative Quot schemes of Kleinian orbifolds \cite{CGGS2} as well as various other moduli spaces which, in many cases, have some relation to the McKay correspondence.

In this paper, we consider Nakajima quiver varieties associated to graphs of affine ADE type. It is known that for generic stability parameters, these spaces are smooth. However, in many cases when these spaces are closely related to moduli spaces such as noncommutative Quot schemes of Kleinian orbifolds or, as a specific case, the Hilbert scheme points of a Kleinian singularity, the quiver moduli spaces are singular. As is well known, for each dimension vector the space of stability parameters has a wall-and-chamber structure; singular spaces correspond to non-generic stability parameters in this decomposition. 

The Euler numbers of smooth quiver varieties of affine ADE type can be calculated by various methods. One possibility is the general method of \cite{hausel2010kac}, but it is difficult to investigate properties of the formulas obtained this way. For a specific open chamber $F_I$ in the space of stability conditions and rank $1$ framing,  the combination of the results of \cite{nakajima2002geometric}, \cite{kuznetsov2007quiver} and the Frenkel--Kac theorem was used in~\cite{gyenge2017euler} to obtain a formula for the generating series in formal variables; see formula~\eqref{eq:orbi_gen} below. 

In this paper, we compute the generating function of Euler characteristics of %a large number of 
Nakajima quiver varieties of affine ADE type corresponding to nongeneric stability conditions lying on the boundary of the chamber $F_I$ in terms of a universal substitution into the generating function attached to the smooth models. 
Our method following Nakajima~\cite{NakajimaEulerNumbers} is to analyze the fibers of the map induced by the variation of the GIT stability condition 
\begin{equation}\label{eq:variationofGIT}
\pi_{\zeta^{\bullet}, \zeta} \colon \NQV_{\zeta}(v,w) \to \NQV_{\zeta^{\bullet}}(v,w),
\end{equation}
when the stability parameter moves from $\zeta$ in the interior of the chamber to $\zeta^{\bullet}$ on a wall. 

Let $I$ be the vertex set of an affine ADE Dynkin diagram $D$. Let 
$$Z_I(w) = \sum_{v \in \N^I} \chi(\NQV_{\zeta}(v,w)) e^{-v}$$ 
be the generating function of Nakajima quiver varieties associated to the diagram $D$ for a fixed general framing vector~$w$, arbitrary dimension vector $v$ and generic stability condition~$\zeta$ (see~\ref{subsec:substitution} for the precise definition of the space where such sums live). The following universal substitution formula, which generalizes and unifies a number of earlier results in \cite{gyenge2018euler,NakajimaEulerNumbers}, is our first main result. 

\begin{theorem}[{Theorem~\ref{thm_sub}}]
\label{thm:Intro1} 
     Let $I^0 \subset I$ be a nonempty, proper subset of the vertex set~$I$ with complement~$I^+=I\setminus I^0$. Let 
     $\zeta^{\bullet}\in F_{I^+}$ be a degenerate stability condition in the stratum of the boundary of the open chamber $F_I$ corresponding to $I^+$; see~\ref{subsec:Nak_q_ADE}.
Let 
\[Z_{I^+}(w)=\sum_{v^+ \in \N^{I^+}} \chi(\NQV_{\zeta^{\bullet}}(v^{\oplus},w)) e^{-v^+},\]
     where, for every $v^+ \in \N^{I^+}$, $v^{\oplus} \in \N^I$ is chosen so that $\NQV_{\zeta^{\bullet}}(v^{\oplus},w)$ is the essentially unique and in a precise sense largest singular quiver variety containing all other $\NQV_{\zeta^{\bullet}}(v',w)$ with $v'|_{I^+} = v^+$ (see Proposition \ref{prop:highest_v} and surrounding discussion).
     Then, with the substitution $s_{I^0}$ defined  
    in Definition~\ref{def:subst},  and a specified complex root of unity $c_{w, I^0}$, we have 
     $$Z_{I^+}(w) = c_{w,I^0}^{-1} \cdot s_{I^0}(Z_I(w)).$$
\end{theorem}

Building on the ideas of~\cite{craw2019punctual} dealing with the case of the Hilbert scheme, in~\cite{CGGS2} the second and third authors with their coauthors define certain Quot schemes of modules on simple quotient surface orbifolds, and identify their underlying reduced subschemes with Nakajima quiver varieties for $w = \Lambda_0$, the zeroth fundamental weight; see Theorem~\ref{main_thm_appendix} below. We can apply Theorem \ref{thm:Intro1} in this context. Let $\emptyset \neq I^0 \subsetneq I$, $I^+=I \setminus I^0$ its complement, and let $$Z_{I^+}(\Lambda_0)= \sum_{v^+ \in \N^{I^+}} \chi \left(\Quot_{I^+}^{v^+}([\C^2/\Gamma]) \right) e^{-v^+}$$ be the generating function of Euler characteristics of the orbifold Quot schemes.

\begin{theorem}[{Corollary~\ref{cor_quot})}]\label{thm:Intro2} We have
     $$Z_{I^+}(\Lambda_0)= (c_{\Lambda^0,I^0})^{-1} \left(\prod_{k = 1}^{\infty} (1-e^{-k \delta|_{I^+}})\right)^{-n-1} \sum_{v \in \Z^{I'}} \exp(-k_{I^+} \cdot v) e^{-v|_{I^+}} e^{-\left(\frac{1}{2} v C' v^T \right) \delta|_{I^+}}.$$ Here $\delta$ is the imaginary affine root, $I' = I \setminus \{0\}$ is the vertex set of the finite type Dynkin diagram, $C'$ is the finite type Cartan matrix, and $k_{I^+} \in  2 \pi \sqrt{-1}\Q^{I'}$ is a specified vector. 
\end{theorem}
This expresses the generating function of Euler characteristics of our Quot schemes as a kind of Jacobi--Riemann type theta-function, twisted with roots of unity. 
After further substituting $e^{-\alpha_i} = q^{\delta_i}$ into the above generating series, we get a power series $Z_{I^+}(q)$ in a formal variable~$q$. % Setting $q = \mathrm{exp}(2\pi i \tau)$, we may regard $Z$ as a function of $\tau \in \mathbb{H}$, where $\mathbb{H}$ is the upper half-plane.

\begin{theorem}[{Theorem~\ref{thm:ZIplusmodular})}]\label{cor:Intro3} The series $Z_{I^+}(q)$ is the $q$-expansion of a meromorphic modular form for some congruence subgroup of $\SL(2, \Z)$.
\end{theorem}
This result confirms a consequence of S-duality~\cite{vafa1994strong} for our singular moduli spaces. In the abelian case, our series specialize to certain partition functions which have already appeared in the physics literature 
as \emph{twisted sectors}~\cite{dijkgraaf2008instantons}; we thus in particular confirm the modularity of these expressions.

In a different direction, it is well known that in the type A case, torus localization often reduces questions concerning quiver varieties to problems about computing generating functions of certain partitions (Young diagrams). We perform this analysis, giving a second proof of Theorem~\ref{thm:Intro2} in Type A and rank $1$ framing. In Proposition~\ref{prop_combinatorial_map}, we show how this sheds light on the combinatorics of sets of torus fixed points collapsing together under the variation of GIT map \eqref{eq:variationofGIT}. We also discuss the fermionic Fock space representation of the affine type A Lie algebra $\widehat{\mathfrak{sl}}_{n+1}$, and %show how one can 
we define certain natural subspaces of this representation whose graded dimension agrees with our generating function $Z_{I^+}$
%the generating functions studied earlier 
(Proposition~\ref{prop_trace}). We finally 
derive in Proposition~\ref{prop_comb_model} 
a natural combinatorial model of highest weight representations of finite-dimensional Lie algebras~$\mathfrak{sl}_{r}$ related to the collapsing fibers; %which 
this model may be new and of independent interest.

The structure of the rest of the paper is as follows. In Section~\ref{sec:Prelim} we review the necessary notations and results on Nakajima quiver varieties and ADE root systems. In Section~\ref{sec:substitution}, we introduce our substititon, prove Theorem~\ref{thm:Intro1}, then apply our result to Quot schemes to deduce Theorem~\ref{thm:Intro2}, and finally we prove Theorem~\ref{cor:Intro3}. In Section~\ref{sec:typeA} we give a second, combinatorial proof to Theorem~\ref{thm:Intro2} in the type A case by investigating the fibers of the variation of GIT map, and discuss representation-theoretic connections. 
Appendix~\ref{sec:appendix_cartan} collects information about Cartan matrices that are required for the evaluation of our formulas.

\subsection*{Acknowledgements} We would like to thank Pavel Etingof, Hiraku Nakajima, Michael Wemyss, Yehao Zhou and Alastair Craw for conversations and comments. We are especially grateful to Søren Gammelgaard for numerous comments and for carefully reading the manuscript. Á.Gy.~was supported by the János Bolyai Research Scholarship of the Hungarian Academy of Sciences and by the National Research, Development and Innovation Fund of Hungary, within the Program of Excellence TKP2021-NVA-02 at the Budapest University of Technology and Economics.

\section{Preliminaries}\label{sec:Prelim}

\subsection{Nakajima quiver varieties}

We fix an undirected graph $D$ with vertex set $I$. 
% the (undirected) Dynkin graph $D$ of an untwisted affine Lie algebra of type ADE. 
We associate to $D$ the double quiver $Q$, which is the directed graph obtained from $D$ by replacing every undirected edge of $Q$ by two inversely directed edges between the same vertices in $Q$. We write $Q = (I, H, s, t)$ (or short, $Q = (I, H)$), where $I$ is the set of vertices, $H$ is the set of directed edges, and $s, t \colon H \to I$ are the assignment of source and target vertex to a given edge.

Let $v, w \in \N^I$.
%(where $\N$ denotes the set of non-negative integers). 
Fix a stability condition $\zeta \in \Q^I$. Depending on this data, Nakajima defines the quiver variety $\NQV_{\zeta}(v,w)$ as the moduli space of $\zeta$-semistable $w$-framed $Q$-modules of dimension $v$. We collect here a few relevant facts about quiver representations and Nakajima quiver varieties; see \cite{NakajimaBranching} for details. 

%\begin{itemize}[topsep=0pt]
    %\item 
    
    A (framed) representation $(V, W, f, a, b)$ (or short, $(V,W)$) of a quiver $Q = (I,H)$ is a collection of finite-dimensional vector spaces $(V_i)_{i \in I}, (W_i)_{i \in I}$ together with linear maps $$f = (f_h \colon V_{s(h)} \to V_{t(h)})_{h \in H}, \quad a = (a_i \colon W_i \to V_i)_{i \in I}, \quad \text{and} \quad b = (b_i \colon V_i \to W_i)_{i \in I}$$
    %If $Q$ is a double quiver associated to some undirected graph, 
    We call a representation of $Q$ a module if $(f,a,b)$ additionally satisfy the relations of the preprojective algebra (the definition of which we will not go into here). Given $\zeta \in \Q^I$, a module $(V,W)$ is called $\zeta$-semistable, if for any subspace $V'$ of $V$, closed under all $f$, we have 
    \begin{equation*} \begin{split}
        V' \subseteq \Ker(b) &\Rightarrow \zeta \cdot \dim(V') \leq 0, \\ V' \supseteq \im(a) &\Rightarrow \zeta \cdot \dim(V') \leq \zeta \cdot \dim(V),
    \end{split}\end{equation*}
    where $\dim(V)$ denotes the vector $(\dim(V_i))_{i \in I}$ and $\zeta \cdot \dim(V) = \sum_{i \in I} \zeta_i \dim(V_i)$. The module $(V,W)$ is called $\zeta$-stable if it is $\zeta$-semistable and the above inequalities are strict for all proper, non-zero subspaces $V'$ satisfying the stated conditions. We assume throughout that $\zeta_i \geq 0$ for all $i \in I$.
%  ; we denote this chamber in the space of stability conditions by $F_I$.

    %\item 
    
    By a submodule of $(V,W)$ we mean a subspace $V'$ as above, i.e.\ closed under $f$, with either $V' \subseteq \Ker(b)$ or $V' \supseteq \im(a)$. In the first case, we view $(V',0)$ as an (unframed) module, and in the second case we view $(V',W)$ as another $W$-framed module.

    %\item 
    
    Modules of a given double quiver admit Harder-Narasimhan filtrations and a notion of S-equivalence. The space $\NQV_{\zeta}(v,w)$ is a coarse moduli space of $\zeta$-semistable modules $(V,W)$ with $\dim(V) = v$ and $\dim(W) = w$, up to S-equivalence. It contains an open subset $\NQV^s_{\zeta}(v,w)$, which is a fine moduli space for $\zeta$-stable modules, up to isomorphism. Furthermore each S-equivalence class of semistable modules contains a unique polystable member which is a direct sum of stable modules, and $\NQV_{\zeta}(v,w)$ is likewise a coarse moduli space of $\zeta$-polystable modules up to isomorphism.

    %\item 
    
\subsection{Nakajima quiver varieties for affine ADE Dynkin diagrams}\label{subsec:Nak_q_ADE}

    From now on, we assume that $Q$ is the double quiver associated to a Dynkin diagram $D$ of affine ADE type, with vertex set $I$. 
    According to the McKay correspondence, to such a diagram $D$ there corresponds a finite subgroup $\Gamma<\SL(2,\C)$, up to conjugation. The vertices of $D$ in turn correspond to the irreducible representations of $\Gamma$; denote these representations $\{\rho_i\colon i \in I\}$, with $\rho_0$ the trivial representation.

    Pick a subset $I^+ \subseteq I$ and let $$F_{I^+} = \{\zeta \in \Q^I \; | \; \zeta_i \geq 0 \; \text{for all} \; i \in I, \; \text{and} \; \zeta_i > 0 \Leftrightarrow i \in I^+\}.$$ 
    Then $F_I$ is an open cone in the space $\Q^I$ of stability conditions; we call a stability condition $\zeta\in F_I$ generic. The sets $F_{I^+}$ for 
    $I^+ \subsetneq I$ form a stratifiction of the boundary of~$F_I$. For any $\zeta \in F_{I^+}$ and $\zeta^{\bullet} \in \overline{F_{I^+}}$, (the closure of $F_{I^+}$ in $\Q^I$; equivalently, $\zeta^{\bullet} \in F_{I^{+'}}$ with $I^{+'} \subseteq I^+$), there is a projective morphism $$\pi_{\zeta^{\bullet}, \zeta} \colon \NQV_{\zeta}(v,w) \to \NQV_{\zeta^{\bullet}}(v,w)$$ which is obtained by GIT degeneration. This morphism is an isomorphism if $\zeta^{\bullet} \in F_{I^+}$ as well. The set of such morphisms is closed under all possible compositions.
    
    The stable part $\NQV^s_{\zeta}(v,w)$ is always smooth. For a generic stability condition $\zeta \in F_I$, $\NQV_{\zeta}(v,w) = \NQV^s_{\zeta}(v,w)$, and is therefore smooth. Any $\pi_{\zeta^{\bullet}, \zeta}$ induces an isomorphism from $\pi_{\zeta^{\bullet}, \zeta}^{-1}(\NQV^s_{\zeta^{\bullet}}(v,w))$ onto its image. At the special stability condition $\zeta = 0 \in F_{\emptyset}$, the quiver variety $\NQV_0(v,w)$ is an affine algebraic variety.
    
    %\item 
    
    Let $\zeta^{\bullet}\in F_{I^+}$ be a non-generic stability condition for $I^+ \subsetneq I$. Let $I^0 = I \setminus I^+$ denote the complement of $I^+$. The direct-sum decomposition of any $\zeta^{\bullet}$-polystable module $(V,W)$ is of the form $$(V,W) = (V',W) \oplus \bigoplus_{i \in I^0} (S_i, 0)^{\oplus v^s_i} \; ,$$ where $(V',W)$ is $\zeta^{\bullet}$-stable, $S_i$ denotes the unframed module which is a one-dimensional vector space at vertex $i$ and zero everywhere else, and $v^s = v - v'$; this follows from \cite[Prop.\ 6.7]{nakajima1994instantons} since $v^s$ is supported on a finite-type Dynkin diagram. This allows us to decompose $\NQV_{\zeta^{\bullet}}(v,w)$ into locally closed strata, each isomorphic to $\NQV^s_{\zeta^{\bullet}}(v',w)$ for some $v' \leq v$ with $v'_i = v_i$ for all $i \in I^+$; here $v' \leq v$ denotes the obvious componentwise partial order on~$\N^I$. Conversely, for any such $\zeta^{\bullet}$-stable module $(V',W)$, and vector $v^s \geq 0$ supported on $I^0$, the direct sum above is a $\zeta$-polystable module. Hence, every such $\NQV^s_{\zeta^{\bullet}}(v',w)$, if non-empty, appears as a stratum in $\NQV_{\zeta^{\bullet}}(v,w)$.
%\end{itemize}

For the rest of this paper, $\zeta \in F_I$ will denote a generic stability condition, and $\zeta^{\bullet}\in F_{I^+}$ a non-generic stability condition, with $I^+ \subsetneq I$.

\subsection{Root systems} Let $\mathfrak{g}$ be the affine Lie algebra associated to $D$, and let $\mathfrak{h} \subseteq \mathfrak{g}$ be a Cartan subalgebra. Let $(\alpha_i \in \mathfrak{h}^{\vee})_{i \in I}$ be the fundamental roots, $\delta \in \mathfrak{h}^{\vee}$ the imaginary root, $(\alpha_i^{\vee} \in \mathfrak{h})_{i \in I}$ the fundamental coroots, and $(\Lambda_i \in \mathfrak{h}^{\vee})_{i \in I}$ the fundamental weights. 
The components of the imaginary root $\delta = \sum_{i \in I} \dim(\rho_i) \alpha_i$ in the root basis coincide with the dimensions of the irreducible representations of the finite group $\Gamma$.

We furthermore denote by $\langle \; , \; \rangle \colon \mathfrak{h}^{\vee} \otimes \mathfrak{h} \rightarrow \C$ the natural pairing, and by $C = (C_{ij})_{i,j \in I}$ the Cartan matrix. We have the following relations: $$\langle \alpha_i, \alpha_j^{\vee} \rangle = C_{ij} \;, \quad \langle \delta, \alpha_j^{\vee} \rangle = 0 \; , \quad \langle \Lambda_i, \alpha_j^{\vee} \rangle = \delta_{ij} \quad \text{for all} \; i,j \in I \; .$$ We can  write $\delta = \sum_{i \in I} a_i \alpha_i$, for some $a_i \in \N$. Also pick a distinguished root $\alpha_0$ (labelled by $i = 0$) satisfying $a_0 = 1$, and introduce the scaling element $d \in \mathfrak{h}$ by the conditions $$\langle \alpha_i, d \rangle = \delta_{i0} \quad \text{and} \quad \langle \Lambda_i, d \rangle = 0  \quad \text{for all} \; i \in I \; .$$ Then $\{\alpha_i^{\vee}\}_{i \in I} \cup \{d\}$ is a basis for $\mathfrak{h}$. As a consequence we also obtain the relation $$\alpha_i = \delta_{i0} \delta + \sum_{j \in I} C_{ij} \Lambda_j \quad \text{for all} \; i \in I \; .$$

Now suppose we are given a subset $I^0 \subsetneq I$ as above. Let $D^0$ denote the (possibly disconnected) Dynkin diagram which is the full subgraph of $D$ on the vertex set $I^0$. The following is clear.
\begin{lemma} $D^0$ is a union of Dynkin diagrams of finite ADE type.
\end{lemma}

Let $\mathfrak{g}^{\text{fin}}$ be the finite-dimensional semisimple Lie algebra associated to $D^0$. Let $\mathfrak{h}^{\text{fin}}$ be a Cartan subalgebra. We denote simple roots and fundamental weights by $\alpha^{\text{fin}}_i$ and $\Lambda^{\text{fin}}_i$, respectively, each defined for $i \in I^0$. We furthermore make the identification $\alpha_i = \alpha^{\text{fin}}_i$ for $i \in I^0$, which embeds $(\mathfrak{h}^{\text{fin}})^{\vee}$ into $\mathfrak{h}^{\vee}$. The associated finite Cartan matrix $C^{\text{fin}} = (C_{ij})_{i,j \in I^0}$ is invertible. The relations we gave in the affine case are the same for $\mathfrak{g}^{\text{fin}}$, except all the terms including $\delta$ or $d$ either become irrelevant or vanish. In particular, we have $$\alpha_i = \sum_{j \in I^0} C_{ij} \Lambda^{\text{fin}}_j \quad \text{for all} \; i \in I^0 \; .$$

This data will be used in the description of the geometry of Nakajima quiver varieties. We identify dimension vectors $v$ with affine roots and framing vector $w$ with affine weights as follows: $$v = \sum_{i \in I} v_i \alpha_i \quad \text{and} \quad w = \sum_{i \in I} w_i \Lambda_i \; .$$ We also consider Nakajima quiver varieties associated to the finite-type Dynkin diagram $D^0$, which we denote by $\NQV^{\text{fin}}_{\zeta}(v^s,w^s)$, where $v^s, w^s \in \N^{I^0}$ (or, by abuse of notation, $v^s \in \N^I$ with non-vanishing coefficients only over $I^0$) and $\zeta \in \Q^{I_0}$ is a stability parameter. Again, we identify $$v^s = \sum_{i \in I^0} v^s_i \alpha^{\text{fin}}_i = \sum_{i \in I^0} v^s_i \alpha_i \quad \text{and} \quad w^s = \sum_{i \in I^0} w^s_i \Lambda^{\text{fin}}_i \; .$$ The following results hold.

\begin{enumerate}
    \item $\NQV^s_{\zeta^{\bullet}}(v,w) \neq \emptyset$ if and only if $w-v \in \mathfrak{h}^{\vee}$ appears as an $I^0$-dominant weight of the highest-weight representation $V(w)$ of $\mathfrak{g}$. Here the difference is taken  after identification with weights, not component-wise, and 
    $I^0$-dominant means that $\langle w-v, \alpha_i^{\vee} \rangle \geq 0$ for all $i \in I^0$. This result is~\cite[Prop.\ 2.30]{NakajimaBranching}.

    \item Let $\zeta \in F_I$ be a generic stability condition for $D$ and $\zeta^{\bullet} \in F_{I^+}$. Let $v = v' + v^s$, $v^s$ supported on $I^0$, determine a stratum isomorphic to $\NQV^s_{\zeta^{\bullet}}(v',w)$ in $\NQV_{\zeta^{\bullet}}(v,w)$.  Let $\zeta|_{I^0}\in \Q^{I_0}$ be the restricted stability condition for $D^0$.    By~\cite[Section 2.7]{NakajimaBranching}, the fiber of $\pi_{\zeta^{\bullet}, \zeta}$ over any point in this stratum is 
    %homotopy equivalent to $\NQV^{\text{fin}}_{\zeta|_{I^0}}(v^s,w^s)$, where $$w^s = \sum_{i \in I^0} \langle w - v', \alpha_i^{\vee} \rangle \Lambda_i^{\text{fin}} \; .$$ This is explained in~\cite[Section 2.7]{NakajimaBranching}.
 isomorphic to a distinguished Lagrangian subvariety ${\mathcal L}^{\text{fin}}_{\zeta|_{I^0}}(v^s,w^s)\subset \NQV^{\text{fin}}_{\zeta|_{I^0}}(v^s,w^s)$, where $$w^s = \sum_{i \in I^0} \langle w - v', \alpha_i^{\vee} \rangle \Lambda_i^{\text{fin}} \; .$$ Further, the inclusion ${\mathcal L}^{\text{fin}}_{\zeta|_{I^0}}(v^s,w^s)\subset \NQV^{\text{fin}}_{\zeta|_{I^0}}(v^s,w^s)$ is a homotopy equivalence by~\cite[Cor.5.5]{nakajima1994instantons}. So as far as the Euler characteristics of fibres are concerned, one can work with the full quiver varieties for finite-type diagrams.  We learned this trick from~\cite[1(iv)]{NakajimaEulerNumbers}.
\end{enumerate}

\section{The substitution formula and its applications}\label{sec:substitution}

\subsection{The substitution rule and its properties}
\label{subsec:substitution}

In order to write down our generating functions in an economical way, we introduce formal variables $e^v$ for all $v \in \mathfrak{h}^{\vee}$, subject to the relations $e^0 = 1$ and $e^{v_1 + v_2} = e^{v_1} e^{v_2}$. We will consider generating functions in which infinitely many $e^v$ appear with non-zero coefficients, keeping in mind that one needs to take care that products or substitutions of such functions are well-defined.

In order to define our substitution of variables, we first introduce some notation.
\begin{itemize}[topsep=0pt]
    \item Let $j \in I^0$. The Dynkin diagram $D^0$ has a connected component $D^{0,j}$ containing~$j$, which is the Dynkin diagram corresponding to a finite-dimensional simple Lie algebra. Let $h_j$ denote the dual Coxeter number of $D^{0,j}$.
    
    \item Recall that $C^\text{fin}$ denotes the Cartan matrix of $D^0$. For $j \in I^0$ define $$c_j \coloneqq \sum_{k \in I^0} ((C^\text{fin})^{-1})_{j,k},$$ the sum of elements of the $j$-th row of the inverse of $C^\text{fin}$.
\end{itemize}

\begin{lemma}\label{lem:h-cj}
    Consider the finite type Dynkin diagram which is the full subgraph on vertex set $I \setminus \{0\}$, and let $h$ be its dual Coxeter number. Then $$h = 1 + \sum_{\lukas 0 \to j \in I \setminus \{0\}} c_j \; ,$$ {\lukas where we sum over all arrows in $Q$ starting at $0$ and ending at a vertex in $I \setminus \{0\}$ . }
\end{lemma}
\begin{proof}
    The equality can be checked case by case, using the data presented in Appendix \ref{sec:appendix_cartan}.
\end{proof}

\begin{definition}\label{def:subst}
    We define the substitution $s_{I^0}$ as follows.
    $$s_{I^0} \colon
    \begin{cases}
        e^{\alpha_i} \mapsto \exp \left(\frac{2 \pi \sqrt{-1}}{h_i + 1} \right), & i \in I^0 \; , \\
        e^{\alpha_i} \mapsto e^{\alpha_i} \prod_{i \to j \in I^0} \exp \left(-c_j \frac{2 \pi \sqrt{-1}}{h_j+1}\right), & i \in I^+ \; .
    \end{cases}$$
    These formulas are then extended to substitutions of the terms $e^{\Lambda^{\text{fin}}_i}$, possibly involving higher-order roots of unity.
\end{definition}

%Well-definedness Lemma now in scrap.

\begin{example}\label{ex_sub}
    Let $I^+ = \{0\}$. Then the subgraph $D^0$ is a connected finite type Dynkin diagram; let $h$ denote its dual Coxeter number. By virtue of Lemma \ref{lem:h-cj}, the substitution in this case specialises to $$s_{I^0} \colon \begin{cases} e^{\alpha_i} \mapsto \exp \left(\frac{2 \pi \sqrt{-1}}{h + 1} \right), & i \neq 0 \; , \\ e^{\alpha_i} \mapsto e^{\alpha_i} \exp \left(\frac{4 \pi \sqrt{-1}}{h+1}\right), & i=0 \; . \end{cases} $$
    
    This formula is equivalent to the substitution subject to the conjecture of the last two authors and N\'{e}methi \cite{gyenge2017euler}, which was proven by Nakajima \cite{NakajimaEulerNumbers}.
\end{example}

Our substitution has the following two properties, which are crucial for our calculations. We first obtain a slight generalisation of Nakajima's quantum dimension result.

\begin{proposition}\label{prop:Property1} For a generic stability condition $\zeta \in F_I$, we have
    $$s_{I^0} \left(\sum_{v^s \in \N^{I^0}} \chi \left(\NQV^{\mathrm{fin}}_{\zeta|_{I^0}}(v^s,w^s) \right) e^{w^s-v^s}\right) = 1 \quad \text{for all} \; w^s \in \N^{I^0}.$$
\end{proposition}
\begin{proof}
    The statement is proved in \cite[Theorem~2]{NakajimaEulerNumbers} for connected finite-type Dynkin diagrams, by interpreting the left-hand side as the quantum dimension of the standard module of the quantum loop algebra over $\mathfrak{g}^{\text{fin}}$.
    
   The statement for disconnected Dynkin diagrams is then obtained as follows. Notice that, since $\zeta|_{I^0}>0$, a $Q^0$-module is $\zeta|_{I^0}$-(semi)stable if and only if it does not contain any proper, nonzero submodules $\subseteq \Ker(b)$, which, in turn, holds if and only if its restrictions to the connected components of $Q^0$ satisfy the same condition. Hence, the quiver variety in this case is just the product of the quiver varieties associated with the connected components. Since the Euler characteristic is multiplicative on products, the generating series above is simply the product of the respective generating series for the connected components.

    %The statement for disconnected Dynkin diagrams is then obtained as follows. The quiver variety in this generality is just the product of the quiver varieties associated with the connected components. In the case of $\zeta|_{I^0}>0$, a $Q^0$-module is $\zeta|_{I^0}$-(semi)stable if and only if it does not contain any proper, nonzero submodules contained in $\Ker(b)$, which, in turn, holds if and only if its restrictions to the connected components of $Q^0$ satisfy the same condition. Since the Euler characteristic is multiplicative on products, the generating series above is simply the product of the respective generating series for the connected components.
\end{proof}

In order to state the second property, we introduce new variables $$\tilde{\alpha}_i \coloneqq \sum_{j \in I^0} C_{ij} \Lambda^{\text{fin}}_j \quad \text{for all} \; i \in I \; .$$ Notice that $\tilde{\alpha}_i = \alpha_i$ for $i \in I^0$ but not for $i \in I^+$.
\begin{proposition}\label{prop:Property2}
    $$s_{I^0} \left(e^{\alpha_i - \tilde{\alpha}_i}\right) = e^{\alpha_i} \quad \text{for all} \; i \in I^+.$$
\end{proposition}
\begin{proof}
    Unravelling the linear relations between $\alpha_i$, $\Lambda_i$, and $\tilde{\alpha}_i$, we obtain $$\alpha_i - \tilde{\alpha}_i = \sum_{j \in I} B_{ij} \alpha_j \; , \quad \text{where} \quad {\lukas B} = \begin{pmatrix} \id_{I^+} & {\lukas \Tilde{B}} \\ 0 & 0 \end{pmatrix} \; ,$$ and ${\lukas \Tilde{B}}$ is a matrix whose rows are indexed by $I^+$ and whose columns are indexed by $I^0$, determined as follows. For $i \in I^+$ the $i$-th row of ${\lukas \Tilde{B}}$ is the sum of $j$-rows of $(C^\text{fin})^{-1}$, where we sum over edges connecting $i$ to some $j \in I^0$. Collecting all the corresponding factors one can now see that $$s_{I^0} \left(e^{\alpha_i - \tilde{\alpha}_i}\right) = e^{\alpha_i} \left( \prod_{i \to j \in I^0}  \exp \left(-c_j \frac{2 \pi \sqrt{-1}}{h_j+1}\right)\right) \left( \prod_{i \to j \in I^0} \exp \left(c_j \frac{2 \pi \sqrt{-1}}{h_j+1}\right) \right)  = e^{\alpha_i} \; .$$
\end{proof}

\subsection{Proof of the universal substitution formula}

The following finiteness result will explain why the substituted generating function is actually well-defined, and what exactly it enumerates.

\begin{proposition}\label{prop:highest_v}
    For any $v^+ \in \N^{I^+}$, there are only finitely many $v \in \N^I$ with $v|_{I^+} = v^+$ and $\NQV^s_{\zeta^{\bullet}}(v,w) \neq \emptyset$. Consequently there exists $v^{\oplus} \in \N^I$ such that
    \begin{enumerate}
        \item $v^{\oplus}|_{I^+} = v^+$, and
        \item any $v \in \N^I$ with $v|_{I^+} = v^+$ and $\NQV^s_{\zeta^{\bullet}}(v,w) \neq \emptyset$ satisfies $v \leq v^{\oplus}$.
    \end{enumerate}
    For any such $v^{\oplus}$ we furthermore have $$\chi(\NQV_{\zeta^{\bullet}}(v^{\oplus},w)) = \sum_{\substack{v \in \N^I \\ v|_{I^+} = v^+}} \chi(\NQV^s_{\zeta^{\bullet}}(v,w)).$$
    % Moreover, any two such $v^{\oplus}$ give rise to isomorphic $\NQV_{\zeta^{\bullet}}(v^{\oplus},w)$.
\end{proposition}
\begin{proof}
    Recall that $\NQV^s_{\zeta^{\bullet}}(v,w) \neq \emptyset$ if and only if $w-v$ is an $I^0$-dominant weight in the highest-weight representation $V(w)$ of $\mathfrak{g}$. $I^0$-dominance of $w-v$ is equivalent to $$\langle w, \alpha_i^{\vee}\rangle \geq \langle v, \alpha_i^{\vee}\rangle = \langle v^+, \alpha_i^{\vee}\rangle + \langle v^0, \alpha_i^{\vee}\rangle \quad \text{for all} \quad i \in I^0 \; ,$$ where $v^0 = v - v^+ \geq 0$ is supported on $I^0$. Since the dimension vectors $v^0$ have only {\lukas non-negative} components, we furthermore have the inequality $\langle v^0, c|_{I^0} \rangle \geq 0$, where $c|_{I^0} = \sum_{i \in I^0} \delta_i \alpha_i^{\vee}$ is the restriction of the central element of $\mathfrak{g}$. Since the coefficients of $c|_{I^0}$ with respect to the $\alpha_i$ are positive, the region in $\R^{I^0}$ cut out by the above inequalities for $v^0$ is bounded (in fact, a simplex), and hence contains only finitely many $v^0$ with integer coefficients. One can then choose $v^{\oplus}$, for example, as the component-wise maximum of all such $v$.

    The second claim follows from the fact {\lukas that} $\NQV_{\zeta^{\bullet}}(v^{\oplus},w)$ is stratified with strata in bijection with and isomorphic to $\NQV^s_{\zeta^{\bullet}}(v,w)$, for $v|_{I^+} = v^+$.

\end{proof}

\begin{remark}\label{rem:canIncrease_voplus}
    {\lukas Note that Proposition~\ref{prop:highest_v} does not state that the stable locus $\NQV^s_{\zeta^{\bullet}}(v^{\oplus},w)$ is non-empty. The vector $v^{\oplus}$ is merely chosen so that $\NQV_{\zeta^{\bullet}}(v^{\oplus},w)$ contains all the non-empty $\NQV^s_{\zeta^{\bullet}}(v,w)$ with $v|_{I^+} = v^+$ as strata. In particular, once a vector $v^{\oplus}$ satisfying Proposition~\ref{prop:highest_v} is found, any other vector $v^{\oplus'} \geq v^{\oplus}$ with $v^{\oplus'}|_{I^+}=v^+$ will satisfy Proposition~\ref{prop:highest_v} as well.}
    
    It seems likely that, up to isomorphism, $\NQV_{\zeta^{\bullet}}(v^{\oplus},w)$ does not actually depend on the choice of $v^{\oplus}$ as long as it satisfies the conditions of Proposition \ref{prop:highest_v}. For any two such dimension vectors $v^{\oplus} \leq v^{\oplus'}$ the map given by adding a $(v^{\oplus'} - v^{\oplus})$-dimensional trivial module to a given $Q$-module induces a closed embedding $\NQV_{\zeta^{\bullet}}(v^{\oplus},w) \rightarrow \NQV_{\zeta^{\bullet}}(v^{\oplus'},w)$; the details of this will be investigated in forthcoming work. From the properties given in Section \ref{sec:Prelim}, we furthermore know that this morphism is a bijection on closed points. Hence, at least the underlying reduced subscheme of $\NQV_{\zeta^{\bullet}}(v^{\oplus},w)$ does not depend on the choice of $v^{\oplus}$.
\end{remark}

Now recall that our framing vector for the fiber of $\pi_{\zeta^{\bullet}, \zeta}$ {\lukas over the stratum  in $\NQV_{\zeta^{\bullet}}(v,w)$ isomorphic to $\NQV^s_{\zeta^{\bullet}}(v',w)$} is defined as $$w^s \coloneqq \sum_{i \in I^0} \langle w - v', \alpha_i^{\vee} \rangle \Lambda_i^{\text{fin}} \; ,$$ which depends on $w$ and $v'$. For brevity, let us denote $w|_{I^0} = \sum_{i \in I^0} w_i \Lambda_i^{\text{fin}}$, where $w = \sum_{i \in I} w_i \Lambda_i$ {\lukas, and $v' = \sum_{i \in I} v'_i \alpha_i$.}

\begin{lemma}\label{lem:ws_from_wrestriction}
    $$w^s = w|_{I^0} - v' + \sum_{j \in I} v'_j(\alpha_j - \tilde{\alpha_j})$$
\end{lemma}
\begin{proof}
\begin{equation*}\begin{split}
    w^s =& \sum_{i \in I^0} \langle w - v', \alpha_i^{\vee} \rangle \Lambda_i^{\text{fin}} \\
    =& w|_{I^0} - \sum_{i \in I^0} \langle v', \alpha_i^{\vee} \rangle \Lambda_i^{\text{fin}} \\
    =& w|_{I^0} - \sum_{i \in I^0, j \in I} v'_j C_{ji} \Lambda_i^{\text{fin}} \\
    =& w|_{I^0} - v' + \sum_{j \in I} v'_j(\alpha_j - \tilde{\alpha_j}) \;.
\end{split}\end{equation*}
\end{proof}

As in the Introduction, fix a framing vector $w \neq 0$ and a generic stability condition $\zeta$. We consider the generating function $$Z_I(w) \coloneqq \sum_{v \in \N^I} \chi(\NQV_{\zeta}(v,w)) e^{-v}$$
of smooth Nakajima quiver varieties at generic stability. 
Our main result is the following.

%where $v'|_{I^+} = \sum_{i \in I^+} v^i \alpha_i$.

\begin{theorem}\label{thm_sub}
    The {\lukas substituted} generating series $s_{I^0}(Z_I(w))$ is well-defined. The substitution formula $$s_{I^0}(Z_I(w)) = c_{w,I^0} \sum_{v^+ \in \N^{I^+}} \chi(\NQV_{\zeta^{\bullet}}(v^{\oplus},w)) e^{-v^+}$$ holds, where for every $v^+ \in \N^{I^+}$, $v^{\oplus} \in \N^I$ is chosen as in Proposition \ref{prop:highest_v}, and $$c_{w, I^0} = s_{I^0}(e^{-w|_{I^0}})$$ is a complex root of unity.
\end{theorem}
\begin{proof}
By means of the stratification of $\NQV_{\zeta^{\bullet}}(v,w)$ and homotopy type of fibers of $\pi_{\zeta^{\bullet}, \zeta}$, we obtain that $$\chi(\NQV_{\zeta}(v,w)) = \sum_{v'} \chi(\NQV^s_{\zeta^{\bullet}}(v',w)) \chi(\NQV^{\text{fin}}_{\zeta|_{I^0}}(v^s,w^s))$$ where the sum runs over all $v'$ for which $\NQV^s_{\zeta^{\bullet}}(v',w) \neq \emptyset$, $v^s=v-v'$, and \[w^s = \sum_{i \in I^0} \langle w - v', \alpha_i^{\vee} \rangle \Lambda_i^{\text{fin}}\] as above. Summing over all $v$, we obtain the following equality of generating functions. $$Z_I(w) = \sum_{v'} \left(\sum_{v^s} \chi(\NQV^{\text{fin}}_{\zeta|_{I^0}}(v^s,w^s)) e^{w^s-v^s}\right) \chi(\NQV^s_{\zeta^{\bullet}}(v',w)) e^{-v'-w^s}.$$ 

Now, by Proposition \ref{prop:highest_v}, we have $$\sum_{\substack{v \in \N^I \\ v|_{I^+} = v^+}} \chi(\NQV^s_{\zeta^{\bullet}}(v,w)) = \chi(\NQV_{\zeta^{\bullet}}(v^{\oplus},w)) \;.$$ Hence, the generating series $\sum_{v'} \chi(\NQV^s_{\zeta^{\bullet}}(v',w)) e^{-v'|_{I^+}}$ is well-defined and equals:
\begin{equation*}\begin{split}
    \sum_{v^+ \in \N^{I^+}} \chi(\NQV_{\zeta^{\bullet}}(v^{\oplus},w)) e^{-v^+} =& \sum_{v'} \chi(\NQV^s_{\zeta^{\bullet}}(v',w)) e^{-v'|_{I^+}} \\ 
    =& \sum_{v'} \chi(\NQV^s_{\zeta^{\bullet}}(v',w)) s_{I^0} \left(e^{-\sum_{j \in I} v'_j(\alpha_j - \tilde{\alpha_j})} \right)\\
    =& s_{I^0} (e^{w|_{I^0}}) s_{I^0}\left(\sum_{v'} \chi(\NQV^s_{\zeta^{\bullet}}(v',w)) e^{-v'-w^s} \right)\\
    =& s_{I^0} (e^{w|_{I^0}}) \cdot \\
    & s_{I^0}\left(\sum_{v'} \left(\sum_{v^s} \chi(\NQV^{\text{fin}}_{\zeta|_{I^0}}(v^s,w^s)) e^{w^s-v^s}\right) \chi(\NQV^s_{\zeta^{\bullet}}(v',w)) e^{-v'-w^s} \right) \\
    =& s_{I^0}(e^{w|_{I^0}}) s_{I^0}(Z_I(w)) \; .
\end{split}\end{equation*}
Here we used Proposition \ref{prop:Property2} in the second, Lemma~\ref{lem:ws_from_wrestriction} in the third, and Proposition \ref{prop:Property1} in the fourth equality.

%We now check that this sum (ignoring empty $\NQV^s_{\zeta^{\bullet}}(v,w)$) is finite, and show that it can be interpreted as the Euler characteristic of a single quiver variety $\NQV_{\zeta^{\bullet}}(v^{\oplus},w)$.
\end{proof}

\subsection{Quot schemes of Kleinian orbifolds}
\label{sec:quot}

%The McKay correspondence classifies, among other things, finite subgroups of $\SL(2,\C)$, up to conjugation, through affine ADE Dynkin diagrams. Let $\Gamma < \SL(2,\C)$ be a finite subgroup and let $D$ by the associated Dynkin diagram. The vertices of $D$ correspond to the irreducible representations of $\Gamma$; denote such a representation by $\rho_i$ for $i \in I$. Their dimensions coincide with the components of the imaginary root: $\delta = \sum_{i \in I} \dim(\rho_i) \alpha_i$. We let $i = 0$ represent the trivial representation, and set our framing vector to $w = \Lambda_0$. THIS WAS SAID BEFORE ALREADT

Recall the finite group $\Gamma < \SL(2,\C)$ associated to our Dynkin diagram $D$ by the MacKay correspondence. It defines the quotient stack $[\C^2/\Gamma]$, the Kleinian orbifold. Given $v \in \N^I$, there is an orbifold Hilbert scheme $\Hilb^v([\C^2/\Gamma])$ which parameterises $\Gamma$-invariant ideals $J$ of $\C[x,y]$ for which the quotient $\C[x,y]/J$ is a finite-dimensional $\Gamma$-module of representation type $\bigoplus_{i \in I} \rho_i^{\oplus v_i}$. For a generic stability condition $\zeta \in F_{I}$, it is well-known that there is {\lukas an isomorphism} $$\NQV_{\zeta}(v, \Lambda_0) \to \Hilb^v([\C^2/\Gamma]).$$ 

More recently, the last two authors together with Craw and Gammelgaard constructed in~\cite{CGGS2} more general orbifold Quot schemes associated to the Kleinian orbifold. For any non-empty subset $I^+ \subseteq I$ and $v^+ \in \N^{I^+}$, there exists a moduli space $\Quot_{I^+}^{v^+}([\C^2/\Gamma])$ of quotients of dimension vector $v^+$ of a certain module $R_{I^+}$ over a noncommutative algebra~$B_{I^+}$.
For $I^+ = I$ the full subset and $v \in \N^I$, we have \[\Quot_I^v([\C^2/\Gamma]) \cong \Hilb^v([\C^2/\Gamma]),\] while for $I^+ = \{0\}$ the distinguished vertex and $v\in\N$, we have \[\Quot_{\{0\}}^v([\C^2/\Gamma]) \cong \Hilb^v(\C^2/\Gamma),\] the Hilbert scheme of points on the singular surface $\C^2/\Gamma$. 

 The general Quot schemes $\Quot_{I^+}^{v^+} ([\mathbb{C}^2/\Gamma])$ can be related to Nakajima quiver varieties as follows. For $\emptyset \neq I^+ \subsetneq I$ pick a non-generic stability condition $\zeta^{\bullet} \in F_{I^+}$, and a dimension vector $v^+ \in \N^{I^+}$.
% and $v^{\oplus} \in \N^{I}$ as in Proposition \ref{prop:highest_v}. 
% Then there is a morphism $$\NQV_{\zeta^{\bullet}}(v^{\oplus}, \Lambda_0) \to \Quot_{I^+}^{v^+}([\C^2/\Gamma]).$$

\begin{theorem}\label{main_thm_appendix}
Let $I^+ \subseteq I$ be a non-empty subset and $v^+ \in \N^{I^+}$ be a dimension vector. The orbifold Quot scheme $\Quot_{I^+}^{v^+} ([\mathbb{C}^2/\Gamma])$ is non-empty if and only if the Nakajima quiver variety  $\mathfrak{M}_{\zeta^{\bullet}}(v,\Lambda_0)$ is non-empty for some vector $v\in \N^I$ satisfying $v|_{I^+}=v^+$. In this case, for a suitable choice of $v^{\oplus} \in \N^I$ with $v^{\oplus}|_{I^+}=v^+$, we have an isomorphism $$\Quot_{I^+}^{v^+} ([\mathbb{C}^2/\Gamma])_{\rm red} \cong  \NQV_{\zeta^{\bullet}}(v^{\oplus},\Lambda_0)_{\rm red},$$ where on both sides we take the reduced scheme structure.
\end{theorem}
One proof of this result was recently given by Craw in \cite[Theorem 1]{craw2024orbifoldquotschemesle}. An independent argument was given in a previous version of this paper. This argument and its further consequences will be discussed in forthcoming work.

Theorem~\ref{main_thm_appendix} implies that our substitution formula applies directly to the setting of Quot schemes.

\begin{theorem}\label{thm:subst_quot}
    For any nonempty, proper subset $I^+ \subset I$, let $I^0 = I \setminus I^+$ and consider the generating function $$Z_{I^+}(\Lambda_0) \coloneqq \sum_{v^+ \in \N^{I^+}} \chi \left(\Quot_{I^+}^{v^+}([\C^2/\Gamma]) \right) e^{-v^+} \; .$$ Then we have the substitution formula $$Z_{I^+}(\Lambda_0) = (c_{\Lambda^0,I^0})^{-1} s_{I^0}(Z_I(\Lambda_0)),$$ where $c_{\Lambda^0,I^0}$ is a complex root of unity and $Z_I(\Lambda_0) = \sum_{v \in \N^I} \chi \left(\Hilb^v([\C^2/\Gamma]) \right) e^{-v}$ is the generating function associated with the full subset $I$.
\end{theorem}

Finally, recall from~\cite{gyenge2017euler} that the full generating function $Z_I(\Lambda_0)$ can be expressed in the form of a Jacobi--Riemann type theta-function as
\begin{equation}\label{eq:orbi_gen}
    Z_I(\Lambda_0) = \left(\prod_{k = 1}^{\infty} (1-e^{-k \delta})\right)^{-n-1} \sum_{v \in \Z^{I'}} e^{-v} e^{-\left(\frac{1}{2} v C' v^T \right) \delta} \; ,
\end{equation}
where $I' = I - \{0\}$, $C'$ is the finite-type Cartan matrix with Dynkin diagram the full subgraph of $D$ on the vertex set $I'$, and we identify $v = \sum_{i \in I'} v_i \alpha_i$. 
%Thus we have determined generating function of the Euler characteristics of all orbifold Quot schemes.
To use our substitution (Definition \ref{def:subst}), define the vector of roots of unity $k_{I^+} \in 2 \pi \sqrt{-1} \Q^I$ as follows: $$(k_{I^+})_i \coloneqq \begin{cases} \frac{2 \pi \sqrt{-1}}{h_i + 1} , & i \in I^0 \; , \\ - \sum_{i \to j \in I^0} c_j \frac{2 \pi \sqrt{-1}}{h_j+1}, & i \in I^+ \; . \end{cases}$$

\begin{lemma}\label{lem:subst_delta}
    We have $$k_{I^+} \cdot \delta \coloneqq \sum_{i \in I} (k_{I^+})_i \dim(\rho_i) = 0.$$ Hence, in the substitution of $e^{-\delta}$, the different roots of unity cancel out:
    $$s_{I^0}(e^{-\delta}) = e^{-\delta|_{I^+}}$$
\end{lemma}
\begin{proof}
    Since $\delta$ is in the kernel of the Cartan matrix, and since the Cartan matrix is symmetric, it suffices to show that $k_{I^+}$ is in the image of the Cartan matrix. In fact, we have $$k_{I^+} = C \, (C^{\text{fin}})^{-1} (k_{I^+}|_{I^0}) \;,$$ where $(C^{\text{fin}})^{-1} (k_{I^+}|_{I^0})$ is viewed as a vector in $\Q^{I}$ with zero entries on $I^+$.
\end{proof}

Applying Theorem \ref{thm:subst_quot} and Lemma \ref{lem:subst_delta}, we obtain the following formula, which gives a Jacobi--Riemann type formula for the generating function of all Quot schemes for a fixed nonempty proper subset $I^+$ of the vertex set~$I$ also. 
\begin{corollary}\label{cor_quot} We have
    $$Z_{I^+}(\Lambda_0) = (c_{\Lambda^0,I^0})^{-1} \left(\prod_{k = 1}^{\infty} (1-e^{-k \delta|_{I^+}})\right)^{-n-1} \sum_{v \in \Z^{I'}} \exp(-k_{I^+} \cdot v) e^{-v|_{I^+}} e^{-\left(\frac{1}{2} v C' v^T \right) \delta|_{I^+}},$$ where $k_{I^+} \cdot v = \sum_{i \in I} (k_{I^+})_i v_i$. 
\end{corollary}

In light of the isomorphism $\Quot_{\{0\}}^v([\C^2/\Gamma]) \cong \Hilb^v(\C^2/\Gamma)$ as well as Example~\ref{ex_sub}, 
for~${I^+}=\{0\}$ this becomes the main conjecture of~\cite{gyenge2017euler}, already proved by Nakajima~\cite{NakajimaEulerNumbers} of course.

\subsection{Modularity}\label{sec:modularity}

%\subsection{Theta functions of lattices}
%\label{subsec:theta}

%Forgetting about the equivariant structure, one can collect the topological Euler characteristics of these moduli spaces into generating functions. The \textit{$G$-fixed generating series} is defined as:
%\[Z_{[X/G]}(q)= 1+\sum_{m=1}^{\infty}  \chi( %\mathrm{Hilb}^m(X)^G) q^m.\]
Let $q$ be a formal variable. Consider the generating series 
\[Z_{I^+}(q) = \sum_{v^+ \in \N^{I^+}} \chi \left(\Quot_{I^+}^{v^+}([\C^2/\Gamma]) \right) q^{v^+},\]
where 
\[q^{v^+}= \prod_{i \in I^+} q^{v_i \dim\, \rho_i}.\]
%This is obtained from $Z_{I^+}(\Lambda_0)$ by substituting $q^{\dim\, \rho_i}$ into $e^{\alpha_i}$, $i \in I^{+}$. 
Setting
$q = \mathrm{exp}(2\pi i \tau)$,
we may regard $Z_{I^+}(q)$ as a function of $\tau \in \mathbb{H}$, where $\mathbb{H}$ is the upper half-plane.
Our next result is the following.

\begin{theorem}\label{thm:ZIplusmodular}
The series $Z_{I^+}(q)$ is, up to a suitable rational power of $q$, the Fourier expansion of a meromorphic modular form of weight $|I|/2$ for some congruence subgroup of $\mathrm{SL}(2,\Z)$ acting on the upper half-plane~$\mathbb{H}$.
\end{theorem}

Let $(\mathbb{Z}^n,Q)$ be a lattice equipped with a positive definite quadratic form. 
%Let $\mathbb{R}^n=\mathbb{Z}^n \otimes {\mathbb{R}}$ and extend $Q$ linearly to a form $Q: \mathbb{R}^n \to  \mathbb{R}$.
% Let moreover ${z} \in \mathbb{R}^n$ be a vector.
Extend $Q$ to a form on $\mathbb{Q}^n=\mathbb{Z}^n \otimes {\mathbb{Q}}$. 
Let moreover ${z} \in \mathbb{Q}^n$ be a vector.
The \emph{shifted theta series} associated with the pair $(Q,{z})$ is
\[ \Theta_{Q,{z}}(q) := \sum_{{k} \in \mathbb{Z}^n} q^{Q({k} + {z})}.\]
It is known that $\Theta_{Q,{z}}(q)$ is a modular form of weight $n/2$ for some congruence subgroup of $\mathrm{SL}(2,\Z)$; see e.g. \cite[Section~13.6]{kac1990infinite}.
Theorem~\ref{thm:ZIplusmodular} is an immediate corollary of the following statement. 
\begin{proposition}
\label{prop:ZIpluslincombtheta}
The series $Z_{I^+}(q)$ is a $\Q$-linear combination of shifted theta series associated with integer valued positive definite quadratic forms on $\mathbb{Z}^n$. That is, there exist $N\geq 1$, $a_i \in \Q$, integer valued positive definite quadratic
forms $Q_i$ on $\mathbb{Z}^n$, and 
% ${z}_i \in \mathbb{R}^n$
${z}_i \in \mathbb{Q}^n$
for  $1 \leq i \leq N$, such that $Z_{I^+}(q)$ can be written as
\[ Z_{I^+}(q)=\sum_{i=1}^Na_i\Theta_{Q_i,{z}_i}(q).  \]
\end{proposition}
\begin{proof} The methods of~\cite[Section~3]{toda2015s} generalise to the current setting. We omit the details.
\end{proof}

\begin{remark} %Eta products form a much investigated and well tractable subset of (quasi-)modular forms \cite{kohler2011eta}. 
The question arises whether our series $Z_{I^+}(q)$ 
can be expressed as an eta product.
Recall that one advantage of expressing a modular form as an eta product is that many specific modular properties of eta products can be calculated easily (the congruence subgroup, order of vanishing at cusps, etc.), see \cite{kohler2011eta}. We performed some numerical calculations in type $A$ for low rank cases; see also \cite[Section~6]{dijkgraaf2008instantons} for such calculations in the physics literature. For the root systems $A_1$ and $A_2$, all Quot scheme generating functions lead to eta products. However, $Z_{A_{3},\{0\}}$ is not expressible as an eta product, and we do not expect eta products to exist for larger root systems.
\end{remark}

\subsection{Higher rank framings}\label{sec:higherrank}

Our universal substitution result Theorem~\ref{thm_sub} holds for arbitrary framing vector~$w$; note that singular quiver varieties with higher-rank framing can sometimes be related to moduli spaces of framed sheaves on certain noncommutative surfaces that can be interpreted as partial resolutions of the scheme ${\mathbb P}^2/\Gamma$ \cite{gammelgaard2024noncommutative, gam}. Our formula is however not really useful for explicit results unless one has some information about the generating function $Z_I(w)$ computed at a generic 
stability condition. For weight vectors other than $w=\Lambda_0$ (and the other fundamental weights related to it by diagram automorphisms of $D$), we are not aware of useful formulae away from type $A$. Nakajima 
in~\cite[Theorem 5.2]{nakajima2002geometric} states a result for the total cohomology of quiver varieties of affine ADE type, which however contains terms involving intersection cohomology spaces which appear hard to control for higher rank framing. A formula of a different nature is presented by Hausel in~\cite[Theorem 1]{hausel2010kac}, which also appears difficult to work with. In the {finite} type $D$ case, see also \cite[Proposition~5.5]{nakajima2003t}.

On the other hand, for affine type $A$, torus localisation arguments closely related to those explained in the next section lead to the explicit formula~\cite[Corollary 4.12]{fujii2005combinatorial} for $Z_I(w)$ for all framing vectors~$w$. This formula is entirely analogous to the Jacobi--Riemann theta-function type formulae above, so we will not repeat it here. Using our universal substitution, for $I^+\subset I$ one can then derive an expression for $Z_{I^+}(w)$ analogous to that of Corollary~\ref{cor_quot}. The results of~\ref{sec:modularity} imply the modularity of the corresponding one-variable series $Z_{I^+}(q)$ for non-generic stability conditions.

\section{Type A: combinatorics and representation theory}
\label{sec:typeA}

In this section, we will assume that the Dynkin diagram $D$ is of affine type A, and we consider the framing vector $w=\Lambda_0$ as in~\ref{sec:quot}. In this case, the generating functions can be studied via the combinatorics of colored partitions, and are related to the Fock space representation of the affine type A Lie algebra. In this section we make this explicit, providing an alternative proof of our substitution result in this case.

\subsection{Torus fixed points}
\label{sec:typeAtorusfixedpoints}

Fix an integer $n$ and let $\Gamma$ be the cyclic subgroup of $\SL(2,\C)$ of order $n+1$. The corresponding Dynkin diagram $D$ has vertex set $I = \{0, \ldots, n\}$ labelled in a cyclic manner. 
%By definition, $\Gamma$ acts on $\C^2$; 
The quotient variety $\C^2/\Gamma$ has an $A_{n}$ singularity at the origin. The action of $\Gamma$ commutes with the natural action of the two-torus $T=(\C^{\times})^2$ on the plane. Hence, $T$ acts on the quotient $\C^2/\Gamma$ and the orbifold $[\C^2/\Gamma]$, as well as the orbifold Hilbert schemes $\Hilb^v([\C^2/\Gamma])$ and the orbifold Quot schemes $\Quot_{I^+}^{v^+}([\C^2/\Gamma])$.

Consider the set ${\mathbb N}\times {\mathbb N}$ of pairs of natural numbers; these are drawn as a set of blocks on the plane, occupying the first quadrant. Label blocks diagonally with $(n+1)$ labels $0, \ldots, n$ as in the figure below; the block with coordinates $(i,j)$ is labelled with $(i-j) \mod (n+1)$. We will call this the \emph{pattern of type~$A_n$}.

\begin{center}
\begin{tikzpicture}[scale=0.6, font=\footnotesize, fill=black!20]
  \draw (0, 0) -- (0,7);
  \foreach \x in {1,2,4,5,6,7,8}
    {
      \draw (\x, 0) -- (\x,6.2);
    }
    \draw (0,0) -- (9,0);
   \foreach \y in {1,2,4,5,6}
    {
         \draw (0,\y) -- (8.2,\y);
    }
    \draw (0.5,0.5) node {0};
    \draw (1.5,0.5) node {1};
    \draw (4.5,0.5) node {$n$$-$$1$};
    \draw (5.5,0.5) node {$n$};
    \draw (0.5,1.5) node {$n$};
    \draw (1.5,1.5) node {0};
    \draw (4.5,1.5) node {$n$$-$$2$};
    \draw (5.5,1.5) node {$n$$-$$1$};
    \draw (6.5,0.5) node {0};
    \draw (7.5,0.5) node {1};
    \draw (6.5,1.5) node {$n$};
    \draw (7.5,1.5) node {0};
    
    \draw (0.5,4.5) node {1};
    \draw (1.5,4.5) node {2};
    \draw (0.5,5.5) node {0};
    \draw (1.5,5.5) node {1};
    \draw(0.5,3) node {\vdots};
    \draw(1.5,3) node {\vdots};
    \draw(3,0.5) node {\dots};
    \draw(3,1.5) node {\dots};
    \draw(8.75,0.5) node {\dots};
    \draw(0.5,6.75) node {\vdots};
\end{tikzpicture}
\end{center}

Let $\Part$ denote the set of partitions. We identify a partition $\lambda=(\lambda_1, \dots, \lambda_k) \in \Part$, where $\lambda_1 \geq \ldots \geq \lambda_k$ are positive integers, with its Young diagram, the subset of ${\mathbb N} \times {\mathbb N}$ which consists of the $\lambda_i$ lowest blocks in column $i-1$. The blocks in the Young diagram inherit the $n+1$ labels from the $A_n$ pattern.
%Let $\CZ_\Delta$ denote the resulting set of $(n+1)$-labelled partitions, including the empty partition. 
For a partition $\lambda\in \Part$, let $\wt_j(\lambda)$ denote the number of blocks in $\lambda$ labelled $j$ in the above pattern, and define the multiweight of $\lambda$ to be  $\mwt(\lambda)=(\wt_0(\lambda), \ldots, \wt_n(\lambda)) \in \N^I$.
%$\mathbf{wt}(\lambda)=(\mathrm{wt}_0(\lambda), \ldots, \mathrm{wt}_r(\lambda))$. In particular, $\mathbf{q}^{\mathbf{wt}(\lambda}=\prod_{c \in C} q_c^{\mathrm{wt}_c(\lambda)}$.

\begin{lemma}[{\cite{ellingsrud1987homology, fujii2005combinatorial}}]\label{lem:hilb_FP}
    %The torus $T$ acts with isolated fixed points on $\Hilb([\C^2/\Gamma])$, parametrized by the set ${\mathcal P}$ of $(r+1)$-labelled partitions. More precisely, for non-negative integers $k_0, \ldots, k_r$ and $\rho=\oplus_{j=0}^r \rho_j^{\oplus k_j}$, the $T$-fixed points on $\mathrm{Hilb}^\rho([\C^2/\Gamma])$ are parametrized by $(r+1)$-labelled partitions of multiweight $(k_0, \ldots, k_r)$.

    For any $v \in \N^I$, the fixed points of the $T$-action on $\Hilb^v([\C^2/\Gamma])$ are in bijection with the set $\{\lambda \in \Part | \mwt(\lambda)=v\}$.
\end{lemma}
\begin{proof}[Idea of proof]
    The $T$-fixed points on $\mathrm{Hilb}([{\lukas \C^2/ \Gamma}])$, which coincide with the $T$-fixed points on $\mathrm{Hilb}(\C^2)$, are the monomial ideals in $\C[x,y]$ of finite colength. The monomial ideals are enumerated in turn by Young diagrams of partitions. The labelling of each block gives the weight of the $\Gamma$-action on the corresponding monomial. 
\end{proof}

%Denote by $$Z_{\Gamma}(\mathbf{q}) = \sum_{\lambda\in \Part} \mathbf{q}^{\mathbf{wt}(\lambda)}$$
%the generating series of $(n+1)$-labelled partitions. 

We immediately deduce the following.
\begin{corollary} \label{cor:Acombinatorial}
    The orbifold generating series of the simple surface singularity of type $A_n$ can be identified with the generating function for $(n+1)$-labeled partitions:
    \begin{equation*} \label{eq:orbi_gen_comb}
        Z_I(\Lambda_0) = \sum_{\lambda \in \Part} e^{-\mwt(\lambda)} \; . %Z_{[\C^2/\Gamma]}(\mathbf{q})=Z_{\Gamma}(\mathbf{q}) \.
    \end{equation*}
\end{corollary}

We now turn to the $T$-fixed points on the Quot schemes $\Quot_{I^+}^{v^+}([\C^2/\Gamma])$. Consider a nonempty subset $I^+ \subseteq I$. A partition $\lambda \in \Part$ will be called $I^+$-generated, if the complement of its Young diagram in $\N \times \N$ can be covered by translates of $\N \times \N$ whose bottom-left corner block is labeled by some $i \in I^+$. Equivalently, $\lambda$ is $I^+$-generated if all its addable boxes (i.e., boxes in $\N \times \N$ that, when added to $\lambda$, again produce the Young diagram of a partition -- those boxes are also called generators) are labeled by an element of $I^+$. Notice that for $I^+ = I$ the full set, every partition is $I$-generated, and for $I^{+'} \subseteq I^+$, every $I^{+'}$-generated partition is $I^+$-generated. We denote the set of $I^+$-generated partitions by $\Part_{I^+}$. We define the $I^+$-weight of a partition $\lambda$ as the restricted multiweight vector $\mwt_{I^+}(\lambda) = (\wt_i(\lambda))_{i \in I^+} \in \N^{I^+}$.

%We now turn to the $T$-fixed points on the Quot scheme. Let $J \subset C$ be a set of indices.
%We define a subset $\mathcal{P}^J$
%of the set of $(r+1)$-labelled partitions $\mathcal{P}$ as follows. 
%For a Young diagram $\lambda$ the set of its \emph{generators} consists of those boxes outside of $\lambda$ which have have two common sides with $\lambda$ or with the axes $x$ and $y$.  Let be $\mathcal{P}^{J}$ be the set of Young diagrams which are generated by boxes with colors only from $J$. Shortly, we will call $\mathcal{P}^{J}$ the set of \emph{$J$-generated} Young diagrams. The $J$-weight of a diagram $\lambda \in \mathcal{P}^J$ is $\mathbf{q}^{\mathbf{wt}_J(\lambda)}=\prod_{c \in J} q_c^{\mathrm{wt}_c(\lambda)}$.
%The $J$-variable generating series of $J$-generated Young diagrams is defined as
%\[ Z^{J}_{\Gamma}(\mathbf{q})=\sum_{\lambda \in  \mathcal{P}^J} \mathbf{q}^{\mathbf{wt}_J(Y)}\;. \]

\begin{example} 
\label{ex:yd1}
Let $n=2$. We indicated the generators on the following $\{0,2\}$-generated Young diagram:
\begin{center}
	\begin{tikzpicture}[scale=0.6, font=\footnotesize, fill=black!20]
	\draw (0, 0) -- (3,0);
	\draw (0,1) --(3,1);
	\draw (0,2) --(3,2);
	\draw (0,3) --(2,3);
	\draw (0,4) --(2,4);
	\draw (0,0) -- (0,4);
	\draw (1,0) -- (1,4);
	\draw (2,0) -- (2,4);
	\draw (3,0) -- (3,2);
	\draw (0.5,0.5) node {0};
	\draw (1.5,0.5) node {1};
	\draw (2.5,0.5) node {2};
	\draw (0.5,1.5) node {2};
	\draw (1.5,1.5) node {0};
	\draw (2.5,1.5) node {1};
	\draw (0.5,2.5) node {1};
	\draw (1.5,2.5) node {2};
	\draw (0.5,3.5) node {0};
	\draw (1.5,3.5) node {1};
	\draw (0.5,4.5) node {2};
	\draw (2.5,2.5) node {0};
	\draw (3.5,0.5) node {0};
	\end{tikzpicture}
\end{center}
Its $\{0,2\}$-weight is $(3,3)$.
\end{example}

\begin{lemma}\label{lem:quot_FP}
    For any $v^+ \in \N^{I^+}$, the fixed points of the $T$-action on $\Quot_{I^+}^{v^+}([\C^2/\Gamma])$ are isolated and in bijection with the set $\{\lambda \in \Part_{I^+} | \mwt_{I^+}(\lambda)=v^+\}$.
    
    %The torus $T$ acts with isolated fixed points on $\mathrm{Quot}(\C^2/\Gamma)$, which are in bijection with the set $\mathcal{P}^J$ of $J$-generated $(r+1)$-labelled partitions. More precisely, for every dimension vector $n_J \in \mathbb{N}^{|J|}$, the $T$-fixed points on $\mathrm{Quot}^{n_J}(\C^2/\Gamma)$ are parametrized by $J$-generated $(r+1)$-labelled partitions $\lambda$ with $J$-weight $\mathbf{wt}_J(\lambda)=n_J$. 
\end{lemma}
\begin{proof}
%The $T$-fixed points of $\mathrm{Quot}_{I^+}([\C^2/\Gamma])$ are the monomial submodules of $R_{I^+}$ of finite colength. Inside $\C[x,y]$, the ideals they generate correspond to partitions which are $I^+$-generated. The vector space $R_{I^+}$ has a basis consisting of monomials with corresponding blocks labelled $i \in I^+$ inside $\C[x,y]$; thus the $i$-codimension of a monomial ideal inside $R_{I^+}$ is simply the number of blocks labelled $i$ in the diagram corresponding to it.
It is easy to check that the $B_{I^+}$-module $R_{I^+}$ defined by~\cite{CGGS2} has a monomial-type basis consisting of elements corresponding to blocks in the type $A_n$ pattern labelled by $i \in I^+$. The $T$-fixed points of $\mathrm{Quot}_{I^+}([\C^2/\Gamma])$ {\lukas correspond} to quotients of $R_{I^+}$ by monomial submodules of finite colength. 
The dimension of the $\rho_i$-isotypic component of such a quotient is precisely the number of blocks labelled $i$ in the diagram corresponding to it. The result follows.
\end{proof}

\begin{corollary}\label{cor:part_An_coarse}
    %Let $[\C^2/G_\Delta]$ be a simple singularity orbifold of type $A$. Then its $J$-coarse generating series can be expressed as
    %\begin{equation*}
     %   Z_{\C^2/\Gamma}^J(\mathbf{q})=Z^{J}_{\Gamma}(\mathbf{q}).
    %\end{equation*}
    For a non-empty subset $I^+ \subseteq I$, the generating series of Euler characteristics of $I^+$-Quot schemes of the Kleinian orbifold $[\C^2/\Gamma]$ of type $A_n$ can be viewed as the generating function for $(n+1)$-labeled $I^+$-generated partitions:
    \begin{equation*} 
        Z_{I^+}(\Lambda_0) = \sum_{\lambda \in \Part_{I^+}} e^{-\mwt_{I^+}(\lambda)} \; . %Z_{[\C^2/\Gamma]}(\mathbf{q})=Z_{\Gamma}(\mathbf{q}) \.
    \end{equation*}
\end{corollary}

We again denote by $\alpha_i, i \in I$ the standard basis elements of $\N^I$. From now on we also assume that $I^+$ is a proper non-empty subset, write $I^0 = I \setminus I^+$ and denote by $D^0$ the full subgraph of the Dynkin diagram on the vertex set $I^0$.

\subsection{Combinatorics of the substitution}\label{sec:substitutionComb}

% {\adam It is enlightening to reformulate the substitution formula (\ref{def:subst}) and the proof of Theorem \ref{thm:subst_comb} by combinatorial methods. As we will see, this gives at the same time a combinatorial model for the fibers of the GIT-degeneration morphisms $\pi_{\zeta^{\bullet}, \zeta}$. As in Section~\ref{sec:typeAtorusfixedpoints}, this reformulation is in terms of partitions and Young diagrams. We expect that these methods work more generally. At least in type $D$, Young walls may replace Young diagrams as in \cite{gyenge2018euler}. The advantage of the combinatorial approach is that using different coloring patterns of Young diagrams it may also generalize to other toric cases beyond type A, when Lie theory is not applicable. See some early steps in this direction in...}

Denote by $\tilde{I^+}$ the preimage of $I^+$ under the covering $\Z \to \{0, \ldots, n\}$ mapping $k$ to $k\mod (n+1)$, and the same for $I^0$. For any $i \in \tilde{I^+}$ we now define $$r_i = \max(\{j \in \Z \; | \; i \leq j \text{ and } \{i+1, \ldots, j\} \subseteq \tilde{I^0}\}) - i,$$ and $$l_i = i - \min(\{j \in \Z \; | \; j \leq i \text{ and } \{j, \ldots, i-1\} \subseteq \tilde{I^0}\}).$$ In other words, $r_i$ is the number of consecutive integers to the right of $i$ which do not belong to $\Tilde{I^+}$ and $l_i$ is the number of consecutive integers to the left of $i$ which do not belong to $\Tilde{I^+}$. Finally, for any $i \in \tilde{I^0}$ we define $n_i$ to be the length of the maximal sequence of consecutive integers in $\tilde{I^0}$ containing $i$. Evidently, the numbers $r_i$, $l_i$, and $n_i$ only depend on the congruence class of $i$ modulo $n+1$, and are hence well-defined also on $I^0$ and $I^+$, respectively.

In type $A$, the connected components of $D^0$ are type A finite Dynkin diagrams, which are connected to $I^+$ only at their endpoints. In finite type $A_{n_i}$, we have $h_i = n_i+1$ and the two occuring constants $c_i$ at the endpoints of this connected component, are equal to $\frac{n_i}{2}$. The constant $c_{\Lambda_0, I^0}$ is $1$ if $0 \in I^+$ and can otherwise be computed as $$c_{\Lambda_0, I^0} = \exp \left(-c_0 \frac{2 \pi \sqrt{-1}}{n_0 + 2}\right) = (-1)^m \exp \left(2 \pi \sqrt{-1} \frac{m(m+1)}{2(n_0 + 2)}\right) \; ,$$ where $m$ is the distance between $0$ and $I^+$ (measured on either of the two sides). Using these calculations, Definition \ref{def:subst} and Theorem \ref{thm:subst_quot} take the following form.

%For example, if we again set $n = 4$ and $I = \{1,2\}$, then $$k_2 = n_3 = n_4 = n_0 = l_1 = 3, \; \text{and} \; l_2 = k_1 = 0.$$ Notice that the width plus the height of each of the gray rectangles in (\ref{ex:young1}) must be $4 = k_2 + 1 = l_1 + 1$. Notice that there are also invisible rectangles of height plus width equal to $1$, corresponding to $k_1 = l_2$.

\begin{theorem}\label{thm:subst_comb}
    Consider the following substitution of variables. $$s_{I^0} \colon e^{\alpha_i} \mapsto \begin{cases} e^{\alpha_i} \exp \left(\frac{2 \pi \sqrt{-1}}{r_i + 2} + \frac{2 \pi \sqrt{-1}}{l_i + 2} \right), & i \in I^+, \\ \exp \left(\frac{2 \pi \sqrt{-1}}{n_i + 2} \right), & i \in I^0. \end{cases}$$ Then the generating function $Z_{I^+}(\Lambda_0)$ of $I^+$-generated $(n+1)$-labeled partitions is obtained from the generating function $Z_{I}(\Lambda_0)$ of all $(n+1)$-labeled partitions by $$Z_{I^+}(\Lambda_0) = c_{I^0}^{-1} s_{I^0}(Z_I(\Lambda_0)) \; ,$$ where $c_{I^0} = c_{\Lambda_0, I^0}$ is a complex root of unity which is determined as follows. Let $m$ denote the minimum of $I^+$ in $I = \{0, \ldots, n\}$, then $$c_{I^0} = (-1)^m \exp \left(2 \pi \sqrt{-1} \frac{m(m+1)}{2(l_m + 2)}\right) \; .$$
\end{theorem}

%For the rest of this section 

In the following paragraphs, we give an alternative, combinatorial proof of this result.
Consider the projection $$\pi_{I^+} \colon \Part \to \Part_{I^+},$$ which associates to $\lambda \in \Part$ the smallest $I^+$-generated Young diagram containing $\lambda$. For example, for $n=4$, $I^+ = \{1,2\}$, it maps
\begin{equation}\label{ex:young1}
    \pi_{I^+} \colon \quad
    \begin{ytableau}
        \none \\ 4 \\ 0 \\ 1&2&3 \\ 2&3&4&0 \\ 3&4&0&1 \\ 4&0&1&2&3 \\ 0&1&2&3&4&0&1
    \end{ytableau}
    \quad \mapsto \quad
    \begin{ytableau}
        *(lightgray) 3 \\ *(lightgray) 4 \\ *(lightgray) 0 \\ 1& 2& *(lightgray) 3 & *(lightgray) 4 \\ 2&3& *(lightgray) 4 & *(lightgray) 0 \\ 3&4&0&1 \\ 4&0&1&2& *(lightgray) 3 & *(lightgray) 4 & *(lightgray) 0 \\ 0&1&2&3&4&0&1
    \end{ytableau} 
\end{equation}
Conversely, the boundary of any partition mapping to the one on the right must run through each of the gray rectangles from the top-left to the bottom-right corner.

\begin{proposition}\label{prop_combinatorial_map}
    Given $v \in \N^{I}$, the morphism $$\Hilb^v([\C^2/\Gamma]) \rightarrow \Quot_{I^+}^{v|_{I^+}}([\C^2/\Gamma])$$ is $T$-equivariant, and hence maps fixed points to fixed points. Under the bijections of Lemmas~\ref{lem:hilb_FP} and~\ref{lem:quot_FP}, this mapping is given by $\pi_{I^+}$.
\end{proposition}

\begin{proof}
Let~$J$ be {\lukas the} monomial ideal in $\C[x,y]$ corresponding to a $T$-fixed point of \[\Hilb^v([\C^2/\Gamma]).\] Take its decomposition \[J=\oplus_{i \in I} J_i\] where $J_i$ is the $\rho_i$-isotypic component of $J$.
Combining~\cite[Section~3, Proposition~4.2 and Theorem~6.9]{CGGS2}, we see that 
the map from the orbifold Hilbert scheme to the Quot scheme associates to~$J$ the point in the Qout scheme represented by the $B_{I^{+}}$-module
\[J^{I^{+}} = \oplus_{i \in I^{+}} J_i.\]
On the level of monomials, taking a $\rho_i$-isotypic subspace corresponds to keeping the boxes labeled $i$. Therefore, $J^{I^{+}}$ is generated by the monomials with labels in $I^{+}$ just outside the Young diagram of $J$. We recover the map~$\pi_{I^+}$.
\end{proof}

%Notice that the substitution for the full subset $I = \{0, \ldots, n\}$ does not alter $Z_n$. Notice also that when we apply the substitution to the product $q_0 \cdots q_n$, the roots of unity cancel out and we obtain $\prod_{i \in I} q_i$.

Denote the $q$-binomial coefficient by 
\begin{equation}\label{def:q_binom}
    \binom{n}{k}_q = \frac{(1-q^n) \cdots (1-q^{n-k+1})}{(1-q^k) \cdots (1-q)} \in \Z[q].
\end{equation}
It is known that $\binom{n}{k}_q$ is the generating function for partitions that are contained in a $k \times (n-k)$-rectangle, weighted by number of boxes. Our proof of Theorem \ref{thm:subst_comb} is based on the following observation.

\begin{lemma}\label{lem:q_binom}
    Let $0 \leq k \leq n$ be integers. If $\xi$ is a primitive $(n+1)$-th root of unity, then substituting $q = \xi$ gives $$\binom{n}{k}_{\xi} = (-1)^k \xi^{-\frac{1}{2} k(k+1)}.$$
\end{lemma}
\begin{proof}
    Since $\xi$ is a primitive $(n+1)$-th root of unity, we can substitute $\xi$ into (\ref{def:q_binom}), and obtain
    \begin{equation*}\begin{split}
        \binom{n}{k}_{\xi} = & \frac{(1-\xi^n) \cdots (1-\xi^{n-k+1})}{(1-\xi^k) \cdots (1-\xi)}\\
        = & \frac{1-\xi^n}{1-\xi} \cdots \frac{1-\xi^{n-k+1}}{1-\xi^k} \\
        = & (-\xi^{-1})(-\xi^{-2}) \cdots (-\xi^{-k}) \\
        = & (-1)^k \xi^{-\frac{1}{2} k(k+1)}
    \end{split}\end{equation*}
\end{proof}

\begin{proof}[Proof of Theorem \ref{thm:subst_comb}]
    We verify the substitution formula along the fibers of $\pi_{I^+}$. Precisely, we show that for any $I^+$-generated partition $\lambda$, $$s_{I^0} \left( \sum_{\mu \in \pi_{I^+}^{-1}(\lambda)} e^{-\mwt(\mu)} \right) = c_{I^0} e^{-\mwt_{I^+}(\lambda)}.$$
    
    For $i \in \Z$, let the \textit{$i$-th diagonal} be the set of boxes with coordinates $(x,y)$ such that $x-y=i$, so boxes on the $i$-th diagonal are labelled $i \mod (n+1)$. We cover the pattern of type $A_n$ by \textit{diagonal strips} consisting of $i$-th diagonals for $i_0 \leq i \leq i_1$, where $i_0,i_1 \in \Tilde{I^+}$ and $i_0+1, \ldots, i_1-1 \in \Tilde{I^0}$. Notice that any box labelled $i \in I^0$ will lie in exactly one of these strips, while any box labelled $i \in I^+$ will lie simultaneously in two neighbouring strips.
    
    Fix the $I^+$-generated partition $\lambda$, then each of the diagonal strips contain exactly one of the gray rectangles as depicted in (\ref{ex:young1}). The partitions in the fiber $\pi_{I^+}^{-1}(\lambda)$ are obtained from $\lambda$ by removing boxes from the gray rectangles. A rectangle which lies on the strip bounded by diagonals $i_0<i_1$ has width plus height equal to $i_1-i_0 = r_{i_0} + 1 = l_{i_1} + 1$. When we apply the substitution all the $e^{-\alpha_i}$ corresponding to labels appearing on boxes in that rectangle are substituted by the same root of unity $\xi = \exp \left(-\frac{2 \pi \sqrt{-1}}{r_{i_0}+2} \right)$. By Lemma \ref{lem:q_binom}, We therefore see that taking the sum over all ways to remove gray boxes from the rectangle and then substituting becomes equivalent to multiplying by the root of unity
    \begin{equation}\label{eq:rect_contr}
        \Delta_{i_0}(\lambda) \coloneqq \binom{r_{i_0} + 1}{j}_{\xi^{-1}} = (-1)^j \xi^{\frac{1}{2} j(j+1)},
    \end{equation}
    where $j$ is either the width or the height of the rectangle. In total, these factors multiply and we obtain $$s_{I^0}\left( \sum_{\mu \in \pi_{I^+}^{-1}(\lambda)} e^{-\mwt(\mu)} \right) = \left(\prod_{i_0 \in \tilde{I^+}} \Delta_{i_0}(\lambda) \right) s_{I^0}(e^{-\mwt(\lambda)}) \; .$$

    Next, since $e^{-\mwt(\lambda)}$ is a monomial in the variables $e^{-\alpha_i}$, and since substitution gets rid of the ones where $i \in I^0$, we have $$s_{I^0}(e^{-\mwt(\lambda)}) = \zeta(\lambda) e^{-\mwt_{I^+}(\lambda)}$$ where $\zeta(\lambda)$ is some root of unity. We can factor $\zeta(\lambda)$ into contributions from the diagonal strips, $\zeta(\lambda) = \prod_{i_0 \in \tilde{I^+}} \zeta_{i_0}(\lambda)$, where $\zeta_{i_0}$ is $\exp \left(-\frac{2 \pi \sqrt{-1}}{r_{i_0}+2} \right)$ to the power the number of boxes in the intersection of $\lambda$ with the diagonal strip bounded to the left by the $i_0$-th diagonal. We obtain $$s_{I^0} \left( \sum_{\mu \in \pi_{I^+}^{-1}(\lambda)} e^{-\mwt(\mu)} \right) = \left(\prod_{i_0 \in \tilde{I^+}} \Delta_{i_0}(\lambda)\zeta_{i_0}(\lambda) \right) e^{-\mwt_{I^+}(\lambda)} \; .$$
    
    {\lukas 
    We now calculate $\Delta_{i_0}(\lambda)\zeta_{i_0}(\lambda)$ for the three possible shapes of strips. For $i \in \Tilde{I^+}$ let $L_i$ be the number of boxes on the $i$-th diagonal contained in $\lambda$.
    
    \begin{enumerate}
        \item \label{item:pf_comb1} $i_0 \geq 0$: \; The following diagram represents the contributions of a strip with $i_0 \geq 0$, $L_{i_0}=5$, $L_{i_1}=2$, and $r_{i_0}=l_{i_1}=6$. Each box shown contributes a factor of $\exp \left(-\frac{2 \pi \sqrt{-1}}{r_{i_0}+2} \right)$. The blue boxes do not belong to $\lambda$ and were added to represent the triangular power in the factor $\Delta_{i_0}(\lambda)$.
        \begin{equation*}\begin{ytableau}
            \none & \none & \none & \none & i_0 & *(lightgray) & *(lightgray) & *(lightgray) & *(lightgray) & *(cyan) & *(cyan) & *(cyan) \\ \none & \none & \none & i_0 && *(lightgray) & *(lightgray) & *(lightgray) & *(lightgray) & *(cyan) & *(cyan) \\ \none & \none & i_0 &&& *(lightgray) & *(lightgray) & *(lightgray) & *(lightgray) & *(cyan) \\ \none & i_0 &&&&&&& i_1 \\ i_0 &&&&&&& i_1
        \end{ytableau}\end{equation*}
        We see that each row has a length of $r_{i_0}+2$, so the roots of unity cancel out row-wise. The only remaining factor is the sign from $\Delta_{i_0}$, so $$\Delta_{i_0}(\lambda)\zeta_{i_0}(\lambda) = (-1)^{L_{i_0} - L_{i_1}}.$$
        
        \item \label{item:pf_comb2} $i_1 \leq 0$: \; The procedure is analogous to case (\ref{item:pf_comb1}), but here we use column-wise cancellation.
        \begin{equation*}\begin{ytableau}
            \none & \none & \none & \none & *(cyan) \\ \none & \none & \none & *(cyan) & *(cyan) \\ \none & \none & *(cyan) & *(cyan) & *(cyan) \\ \none & i_0 & *(lightgray) & *(lightgray) & *(lightgray) \\ i_0 && *(lightgray) & *(lightgray) & *(lightgray) \\ && *(lightgray) & *(lightgray) & *(lightgray) \\ && *(lightgray) & *(lightgray) & *(lightgray) \\ &&&& i_1 \\ &&& i_1 \\ && i_1 \\ & i_1 \\ i_1
        \end{ytableau}\end{equation*}
        We obtain $$\Delta_{i_0}(\lambda)\zeta_{i_0}(\lambda) = (-1)^{L_{i_1} - L_{i_0}}.$$
        
        \item \label{item:pf_comb3} $i_0 < 0 < i_1$: \; In this case the roots of unity do not cancel along the columns (or rows), but instead give a fixed root of unity which only depends on $I^+$ and not on $\lambda$. In the diagram below the inverse of this root of unity is represented by the black boxes, which are again not part of $\lambda$.
        \begin{equation*}\begin{ytableau}
            \none & \none & \none & \none & *(cyan) \\ \none & \none & \none & *(cyan) & *(cyan) \\ \none & \none & *(cyan) & *(cyan) & *(cyan) \\ \none & i_0 & *(lightgray) & *(lightgray) & *(lightgray) \\ i_0 && *(lightgray) & *(lightgray) & *(lightgray) \\ && *(lightgray) & *(lightgray) & *(lightgray) \\ && *(lightgray) & *(lightgray) & *(lightgray) \\ &&&& i_1 \\ &&& i_1 \\ *(black) & *(black) & *(black) \\ *(black) & *(black) \\ *(black)
        \end{ytableau}\end{equation*}
        We obtain $$\Delta_{i_0}(\lambda)\zeta_{i_0}(\lambda) = (-1)^{L_{i_1} - L_{i_0} + i_1} \exp \left( 2 \pi \sqrt{-1} \frac{i_1(i_1+1)}{2(r_{i_0} + 2)}\right) \; .$$ Notice that this case occurs at most once in the partition, and when it does $i_1$ is the minimal positive element of $\Tilde{I^+}$. 
    \end{enumerate}
}
    Now $m$, as defined in the statement of the Theorem, is the $i_1$ of case (\ref{item:pf_comb3}) or zero if this case does not occur. The signs $(-1)^{L_{i_0}-L_{i_1}}$ from cases (\ref{item:pf_comb1}) telescope to give $(-1)^{L_m}$ and so do the signs $(-1)^{L_{i_1}-L_{i_0}}$ from cases (\ref{item:pf_comb2}) and (\ref{item:pf_comb3}), so they cancel out. What remains is $$\prod_{i_0 \in \Tilde{I^+}} \Delta_{i_0}(\lambda)\zeta_{i_0}(\lambda) = (-1)^m \exp \left(2 \pi \sqrt{-1} \frac{m(m+1)}{2(l_m + 2)}\right) = c_{I^0} \; ,$$ thus concluding the proof of Theorem \ref{thm:subst_comb}.
\end{proof}

\subsection{Representation theory}

Continuing to work with the Type A case, we finally explain how the generating functions studied above arise as graded dimensions of certain naturally defined vector subspaces of Fock space. We recall that a version of (Fermionic) Fock space is the infinite-dimensional vector space
\[ \BF = \bigoplus_{Y\in\CP} \C \vect{Y}
\]
generated by a basis indexed by the set  $\CP$ of all partitions, which we are going to represent as before by their Young diagrams. 

We will initially use a labelling of diagrams $Y\in\CP$ by the set of integers $\Z$, assigning to a block $s \in Y$ in the $i$-th row and $j$-th column the label $l(s)=j-i\in\Z$. This rule
labels the boxes of a partition diagonally with infinitely many labels. As in our earlier discussion, for $c\in\Z$ we have the $c$-weight map $\wt_c\colon \CP\to\N$ , 
with $\wt_c(Y)$ the number of blocks $s \in Y$ of label~$c$. For a partition $Y\in\CP$, we call a block $s \in Y$ removable, if $Y'=Y\setminus\{s\}\in\CP$; we call a block $t \not\in Y$ addable,  if  $Y'=Y\cup\{t\}\in\CP$. For $Y\in\CP$, we denote by $h_c(Y)\in\{\pm1,0\}$ the difference between the number of addable and removable blocks of label $c$. 

Define four sets of operators on $\BF$, indexed by $c\in\Z$, as follows:
%\[\begin{array}{rcl} E_c \vect{Y} & = &\displaystyle\sum_{\substack{{Y'=Y\setminus\{s\}} \\ {l(s)=c}}}  \vect{Y'}, \\[2.5em]
%F_c  \vect{Y} & = & \displaystyle\sum_{\substack{{Y'=Y\cup\{s\}} \\ {l(s)=c}}} \vect{Y'}, \\[2.5em] 
%H_c \vect{Y} & = & h_c(Y)\vect{Y}, \\[0.5em]
%D_c \vect{Y} & = & \wt_c(Y)\vect{Y}.\end{array}
%\]
\[\begin{array}{c} E_c \vect{Y}  = \displaystyle\sum_{\substack{{Y'=Y\setminus\{s\}} \\ {l(s)=c}}}  \vect{Y'}, \ \ \ \  
F_c  \vect{Y}  =  \displaystyle\sum_{\substack{{Y'=Y\cup\{s\}} \\ {l(s)=c}}} \vect{Y'}, \\[2.5em] 
H_c \vect{Y}  =  h_c(Y)\vect{Y}, \ \ \ \ 
D_c \vect{Y} =  \wt_c(Y)\vect{Y}.\end{array}
\]
It is easy to check by direct computation that these operators satisfy Serre-type commutation relations
%\begin{eqnarray*}
%\[\begin{array}{lcl} 
%\left[H_c, H_{c'}\right] & = & 0, \\[0.2em]
%\left[E_c, F_{c'}\right] & = & \delta_{cc'} H_c, \\[0.2em]
%{\rm ad}(E_c)^{1-a_{cc'}}(E_{c'}) & = &0 \mbox{ for } c\neq c',  \\[0.2em]
%{\rm ad}(F_c)^{1-a_{cc'}}(F_{c'}) & = & 0 \mbox{ for } c\neq c',  \\[0.2em]
%\left[H_c, E_{c'}\right] & = & a_{cc'} E_{c'},  \\[0.2em]
%\left[H_c, F_{c'}\right] & = & -a_{cc'} F_{c'},  \\[0.2em]
%\left[D_c, D_{c'}\right] & = & 0, \\[0.2em]
%\left[D_c, H_{c'}\right] & = & 0, \\[0.2em]
%\left[E_c, D_{c'}\right] & = & \delta_{cc'} E_c,\\[0.2em]
%\left[F_c, D_{c'}\right] & = & -\delta_{cc'} F_c
%\end{array}\]
\[\begin{array}{c} 
\left[H_c, H_{c'}\right] = 0, \ \ \left[E_c, F_{c'}\right] = \delta_{cc'} H_c, \ 
\left[H_c, E_{c'}\right]  =  a_{cc'} E_{c'}, \ \ 
\left[H_c, F_{c'}\right] =  -a_{cc'} F_{c'},  \\[0.2em]
{\rm ad}(E_c)^{1-a_{cc'}}(E_{c'})  = 0 \mbox{ for } c\neq c',  \ \ 
{\rm ad}(F_c)^{1-a_{cc'}}(F_{c'}) =  0 \mbox{ for } c\neq c', \end{array}\]
as well as grading relations 
\[\begin{array}{c} 
\left[D_c, D_{c'}\right] =  0, \ \ 
\left[D_c, H_{c'}\right]  =  0, \ \ 
\left[E_c, D_{c'}\right]  =  \delta_{cc'} E_c,\ \ 
\left[F_c, D_{c'}\right]  =  -\delta_{cc'} F_c,
\end{array}\]
%\end{eqnarray*}
for $c,c'\in\Z$, where
\[ a_{cc'} = \left\{\begin{array}{ll} 2 & \mbox{if } c=c', \\ -1 & \mbox{if } |c-c'|=1, \\ 0 &  \mbox{otherwise.}\end{array}\right. \]
In other words, $\BF$ is a representation of the algebra 
\[ {\mathfrak a} =\langle E_c, F_c, H_c, D_c \colon c\in\Z\rangle 
\] 
with the above set of relations. This is a version of ${\mathfrak{sl}}_\infty$, the Lie algebra defined by the root system $A_\infty$, extended by an infinite set of grading operators $\{D_c\}$. It is clear from the definitions that $\BF$ is an irreducible ${\mathfrak a}$-module.

Next fix a natural number $n\geq 1$, and let us also consider, for $[c]\in \Z/(n+1)\Z$, operators
%\[\begin{array}{rcl} e_{[c]} & = & \displaystyle\sum_{c\equiv [c] \!\!\!\!\mod n+1}  E_c, \\[1.5em]
% f_{[c]} & = & \displaystyle\sum_{c\equiv [c] \!\!\!\!\mod n+1}  F_c, \\[1.5em]
%h_{[c]} & = & \displaystyle\sum_{c\equiv [c] \!\!\!\!\mod n+1}  H_c. \\[1.5em]
%d_{[c]} & = & \displaystyle\sum_{c\equiv [c] \!\!\!\!\mod n+1}  D_c. \\[1.5em]
%\end{array}\]
\[\begin{array}{c} e_{[c]} = \displaystyle\sum_{c\equiv [c] \!\!\!\!\mod n+1}  E_c, \ \ \ \ 
 f_{[c]}  =  \displaystyle\sum_{c\equiv [c] \!\!\!\!\mod n+1}  F_c, \\[1.5em]
h_{[c]}  =  \displaystyle\sum_{c\equiv [c] \!\!\!\!\mod n+1}  H_c. \ \ \ \ 
d_{[c]} =  \displaystyle\sum_{c\equiv [c] \!\!\!\!\mod n+1}  D_c. \\[1.5em]
\end{array}\]
An alternative, equivalent definition of these operators first considers the labelling of $Y\in\CP$ by $n$ labels, applied diagonally and periodically, as in the earlier sections.  

The operators $e_{[c]}, f_{[c]}, h_{[c]}, d_{[c]}$ still act on $\BF$, since for every fixed $Y$, the number of terms in each sum will become finite. 
In particular, the $d_{[c]}$-eigenvalue of $\vect{Y}\in\BF$ is the number $\wt_{[c]}(Y)$ of $[c]$-coloured blocks in $Y$ under the periodic diagonal labelling. 
In this way, as is well known~\cite{kac1990infinite}, we get the algebras
\[ \langle e_{[c]}, f_{[c]}, h_{[c]} \colon [c]\in {\lukas \Z/(n+1)\Z} \rangle \cong \widehat{\mathfrak{sl}}'_{n+1},
\] 
the derived algebra of the affine Lie algebra attached to {\lukas $\widehat A_n$}, and the full affine Lie algebra 
\[ \langle d_{[0]}; e_{[c]}, f_{[c]}, h_{[c]} \colon [c] {\lukas \in\Z/(n+1)\Z} \rangle \cong \widehat{\mathfrak{sl}}_{n+1}.
\] 
Note that $\BF$ is reducible as an $\widehat{\mathfrak{sl}}_{n+1}$-module, though it becomes irreducible if one introduces a further set of operators forming a Heisenberg algebra, leading to a representation of the larger algebra $\widehat{\mathfrak{gl}}_{n+1}$. On the other hand, we get a subalgebra inside $\widehat{\mathfrak{sl}}_{n+1}$, the algebra
\[ \langle e_{[c]}, f_{[c]}, h_{[c]} \colon [c]\in\Z/(n+1)\Z \setminus \{[0]\} \rangle \cong {\mathfrak{sl}}_{n+1},
\] 
corresponding to the inclusion of the root system of finite type $A_n$ into the affine root system.

The inclusion ${\mathfrak{sl}}_{n+1}\hookrightarrow \widehat{\mathfrak{sl}}_{n+1}$ has an analogue at the level of
the infinite root system: define the subalgebra
\[ {\mathfrak a_0} =\langle E_c, F_c, H_c, D_c \colon c\in\Z, c \neq 0 \!\!\!\!\mod n+1\rangle \hookrightarrow {\mathfrak a}.
\]
It is clear that ${\mathfrak a_0}$ is isomorphic to an infinite direct sum, indexed by $\Z$, of copies of~${\mathfrak{sl}}_{n+1}$. 

Let us connect this discussion with the ideas in earlier sections. 
% Assume that $\Gamma<{\rm SL}(2,\C)$ is abelian of type $A_n$. 
We proceed to show that several of the generating functions studied earlier can be written as traces of grading operators on Fock space and certain subspaces. First of all, we have the orbifold generating series
$Z_I(\Lambda_0)$; by Corollary~\ref{cor:Acombinatorial} agrees with the generating function of diagonally {\lukas $(n+1)$}-labelled partitions. The generating function $Z_{\{0\}}(q)$ of Euler characteristics of Quot schemes of $[\C^2/\Gamma]$ corresponding to the subset $I^+=\{0\}$, in other words to Hilbert schemes of points of the singular surface $\C^2/\Gamma$, is given by the generating function of $0$-generated partitions $\CP_0$ in the periodic labelling. By Corollary~\ref{cor:part_An_coarse}, this generalises to an arbitrary non-empty subset $I^+\subset I$. 

In our context, define 
%\[ \BF_0 = \bigcap_{c\not\equiv 0 \!\!\!\!\mod (r+1)} \ker (F_c) \subset \BF.\] 
\[ \BF_{I^+} = \bigcap_{\lukas (c\!\!\!\!\mod (n+1))\not\in I^+} \ker (F_c) \subset \BF,
\] 
It is then clear by the definition of an $I^+$-generated partition that the vector space $\BF_{I^+}$ has a natural basis consisting of vectors~$|T\rangle$ for $Y\in\CP_{I^+}$. An important point is that $\BF_{I^+}$ cannot be defined in terms
of the algebra of periodicised operators $\widehat{\mathfrak{sl}}_{n+1}$, but only in terms of operators belonging to the much larger algebra~${\mathfrak a}$. 
On the other hand, it follows from the commutation relations of ${\mathfrak a}$ that the operators $D_c$ for $c\mod (n+1)\in I^+$, and thus the operators $d_i$ for $i\in I^+$, still act on $\BF_{I^+}$. 

\begin{proposition}\label{prop_trace} We have the identity
\[Z_{I^+}(\Lambda_0) = {\rm tr_{\lukas \BF_{I^+}}} e^{-{\bf d}_{I^+}},
\]
with ${\bf d}_{I^+} = \sum_{i\in I^+}d_{i} \alpha_i$. 
\end{proposition}

%Indeed, it is clear from the definitions that, with  \[Z_{\{0\}}(\Lambda_0) = {\rm tr_{\BF_0}} e^{-d_{[0]\alpha_0}}.\]

%More generally, for any subset $I^+\subset\{0,\ldots, r\}$, we considered the generating function $Z_{I^+}(\Lambda_0)$, the generating function of $I^+$-generated partitions. The corresponding subspace of $\BF$ is and we get, similarly to the previous statement, that

Proposition~\ref{prop_trace} ``categorifies'' the Euler characteristic calculations to computing graded dimensions of naturally occuring vector spaces. One potential application is to consider the residual symmetries of these vector spaces. 

Indeed, starting with the case $I^+=\{0\}$, define
\[  {\mathfrak c}_{\{ 0\}}  = C_{{\mathfrak a}, \BF}({\mathfrak a}_0)
\]  
to be the subalgebra of ${\mathfrak a}$ of operators that commute with the action of ${\mathfrak a}_0$ on $\BF$. This Lie subalgebra ${\mathfrak c}_{\{ 0\}}\subset {\mathfrak a}$ clearly acts
on the space $\BF_0$. However, it is easy to see that this algebra is rather small and in particular abelian, as the decomposition of $\BF$ as an ${\mathfrak a}_0$-module is 
multiplicity free. So it is not clear that this is a particularly fruitful idea to study further. 
Note that the commutant ${\mathfrak w}$ of $\mathfrak{sl}_{n+1}$ in 
$\widehat{\mathfrak{sl}}_{n+1}$ in its action on $\BF$ is a much larger, interesting algebra that contains~${\mathfrak c}_{\{ 0\}}$: investigated first by Frenkel, ${\mathfrak w}$~is a version of the $W$-algebra ${\mathcal W}(\widehat{\mathfrak{sl}}_{n+1})$.

Using the same idea, one can more generally define slightly larger subalgebras ${\mathfrak c}_{I^+}$ of the full algebra ${\mathfrak a}$ that may be of interest. 

We finally give a representation-theoretic meaning to the partitions in {\lukas rectangles} which appear in the enumeration of the fibers, as described in~\ref{sec:substitutionComb}. In the pattern of type $A_n$, consider a rectangle $R_{a,b,c}$ with sides parallel to the coordinate axes, such that:
\begin{enumerate}
    \item $R_{a,b,c}$ intersects at most one diagonal of any given label,
    \item the {\lukas box in the} top-left corner of $R_{a,b,c}$ is labeled $a$,
    \item the {\lukas box in the} bottom-left corner of $R_{a,b,c}$ is labeled $b$, and
    \item the {\lukas box in the} bottom-right corner of $R_{a,b,c}$ is labeled $c$.
\end{enumerate}
Notice that these conditions determine the width and height of the rectangle, as well as the labelling of its boxes uniquely, and also force $b$ to lie between $a$ and $c$ along the cyclic labeling of the Dynkin diagram. We consider now the set $\Part_{a,b,c}$ of Young diagrams contained in $R_{a,b,c}$, and $$\BF_{a,b,c} \coloneqq \bigoplus_{Y \in \Part_{a,b,c}} \C \vect{Y} \; .$$

Now, consider the finite type A Dynkin diagram $D^{a,c}$ that is the subgraph of the affine Dynkin diagram going from vertex $a$ to vertex $c$ along the cyclic labeling (in particular, via~$b$). Let $I^{a,c}$ denote its vertex set, and $h$ its cardinality. For any vertex $i \in I^{a,c}$ we define operators $e_i, f_i, h_i$ acting on $\BF_{a,b,c}$ exactly as for general partitions. The relations between these generators are the same as before, so we have an algebra $$\langle e_i, f_i, h_i \; | \; i \in I^{a,c} \rangle \simeq \mathfrak{sl}_{h+1}.$$ 

\begin{proposition}\label{prop_comb_model}
     The space $\BF_{a,b,c}$, as a representation of $\mathfrak{sl}_{h+1}$, is isomorphic to the irreducible fundamental highest weight representation $V(\Lambda_{b})$.
\end{proposition}
\begin{proof}
    Any partition inside $R_{a,b,c}$ can be obtained from the empty partition $Y_0$ by repeatedly adding boxes. Hence, $\BF_{a,b,c} = \mathfrak{n} Y_0$, where $\mathfrak{n} = \langle f_i \; | \; i \in I^{a,c} \rangle$ is the negative triangular direct summand of $\mathfrak{sl}_{h+1}$. Furthermore, since $Y_0$ has no removable boxes, and only one addable box, which is labeled $b$, we have $$h_i Y_0 = \delta_{bi} Y_0 = \langle \Lambda_b, h_i \rangle Y_0,$$ so $Y_0$ has weight $\Lambda_b$. The assertion follows, since we know that the dimension of $\BF_{a,b,c}$ agrees with that of $R_{a,b,c}$: $$\dim \BF_{a,b,c} = |\Part_{a,b,c}| = \binom{\text{width plus height of} \; R_{a,b,c}}{\text{width of} \; R_{a,b,c}} = \dim(V(\Lambda_b)) \; ,$$ see, e.g.\ \cite[Prop.\ 13.2]{carter2005lie}.
\end{proof}

We finally remark that according to Nakajima~\cite{NakajimaEulerNumbers}, the type A fundamental representations $V(\Lambda_b)$ are the same as the so-called standard modules of $\mathfrak{sl}_{h+1}$, whose underlying vector spaces are the cohomology spaces $\bigoplus_{v^s} H^*(\NQV^{\text{fin}}_{\zeta|_{I^0}}(v^s,\Lambda_b))$. Nakajima's central result (cf.\ Proposition \ref{prop:Property1}) is that the quantum dimensions of these standard modules are equal to $1$. His proof in type A is essentially the same as that of our Lemma \ref{lem:q_binom}.

\appendix
\clearpage
\section{Inverses of ADE Cartan Matrices}\label{sec:appendix_cartan}

We collect here data about finite-type ADE Cartan matrices and their inverses, entries of which appear in our substitution formula, Definition \ref{def:subst} (see e.g.\ \cite[Reference Chapter \S 2]{OnishchikVinberg}). We let $h$ denote the dual Coxeter number of each Dynkin diagram. 

\vspace{0.1in}

{
For type $A_l$, we have, $h = l+1$, and for $1\leq i,j\leq l$,
\begin{small}
\[C_{ij} = \left\{\begin{array}{ll} 2, & i = j, \\ -1, & |i-j| = 1, \\ 0, & \text{otherwise.}\end{array}\right. \ \  
        (C^{-1})_{ij} = \min\{i,j\} - \frac{ij}{l+1} \; .\]
\end{small}

For type $D_l$, we have, $h = 2l-2$, and for $1\leq i,j\leq l$,
\begin{small}
\[   C_{ij} = \left\{ \begin{array}{ll} 2, & i = j, \\ -1, & |i-j| = 1 \; \text{and} \; i,j \leq l-1, \\ -1, & (i,j) = (l-2,l) \; \text{or} \; (l,l-2), \\ 0, & \text{otherwise.} \end{array} \right. \ \ 
        (C^{-1})_{ij} = \left\{\begin{array}{ll} \min\{i,j\}, & i,j \leq l-2, \\ \frac{i}{2}, & i \leq l-2, j \geq l-1, \\ \frac{j}{2}, & i \geq l-1, j \leq l-2, \\ \frac{1}{4}(l - 2|i-j|), & i,j \geq l-1. \end{array}\right.  
\]
\end{small}

For type $E_6$, we have $h = 12$, and
\begin{small}
\[    C = \begin{pmatrix} 2&-1&&&& \\ -1&2&-1&&& \\ &-1&2&-1&&-1 \\ &&-1&2&-1& \\ &&&-1&2& \\ &&-1&&&2 \end{pmatrix}, \ \  
        C^{-1} = \frac{1}{3} \begin{pmatrix} 4&5&6&4&2&3 \\ 5&10&12&8&4&6 \\ 6&12&18&12&6&9 \\ 4&8&12&10&5&6 \\ 2&4&6&5&4&3 \\ 3&6&9&6&3&6 \end{pmatrix} \; . \]
\end{small}

For type $E_7$, we have $h = 18$, and
\begin{small}
    \[        C = \begin{pmatrix} 2&-1&&&&& \\ -1&2&-1&&&& \\ &-1&2&-1&&& \\ &&-1&2&-1&&-1 \\ &&&-1&2&-1& \\ &&&&-1&2& \\ &&&-1&&&2  \end{pmatrix},\ \  
        C^{-1} = \frac{1}{2} \begin{pmatrix} 3&4&5&6&4&2&3 \\ 4&8&10&12&8&4&6 \\ 5&10&15&18&12&6&9 \\ 6&12&18&24&16&8&12 \\ 4&8&12&16&12&6&8 \\ 2&4&6&8&6&4&4 \\ 3&6&9&12&8&4&7 \end{pmatrix} \; . \]
\end{small}

Finally for type $E_8$, we have $h = 30$, and
\begin{small}
\[
        C = \begin{pmatrix} 2&-1&&&&&& \\ -1&2&-1&&&&& \\ &-1&2&-1&&&& \\ &&-1&2&-1&&& \\ &&&-1&2&-1&&-1 \\ &&&&-1&2&-1& \\ &&&&&-1&2& \\ &&&&-1&&&2  \end{pmatrix}, \ \ 
        C^{-1} = \begin{pmatrix} 2&3&4&5&6&4&2&3 \\ 3&6&8&10&12&8&4&6 \\ 4&8&12&15&18&12&6&9 \\ 5&10&15&20&24&16&8&12 \\ 6&12&18&24&30&20&10&15 \\ 4&8&12&16&20&14&7&10 \\ 2&4&6&8&10&7&4&5 \\ 3&6&9&12&15&10&5&8  \end{pmatrix} \;. \]
\end{small}

%\printbibliography

\bibliographystyle{amsplain}
\bibliography{References}

\end{document}